\documentclass[preprint,12pt]{elsarticle}

\usepackage[a4paper,margin=1in]{geometry}
\usepackage{amsmath,amssymb,amsthm,mathtools}
\usepackage{booktabs,array}
\usepackage{enumitem}
\usepackage{microtype}
\usepackage[colorlinks=true,linkcolor=blue,citecolor=blue,urlcolor=blue]{hyperref}

\numberwithin{equation}{section}

\newtheorem{theorem}{Theorem}[section]
\newtheorem{proposition}[theorem]{Proposition}
\newtheorem{lemma}[theorem]{Lemma}
\newtheorem{corollary}[theorem]{Corollary}
\newtheorem{remark}[theorem]{Remark}

\newcommand{\R}{\mathbb R}
\newcommand{\Sph}{\mathbb S}

\newcommand{\norm}[2]{\left\|#1\right\|_{#2}}

\newcommand{\pa}{\partial}
\newcommand{\grad}{\nabla}
\newcommand{\Div}{\operatorname{div}}
\newcommand{\curl}{\operatorname{curl}}
\newcommand{\be}{\beta}
\newcommand{\ta}{\tau}
\newcommand{\nv}{\nu}

\newcommand{\F}{\mathcal F}
\newcommand{\D}{\mathcal D}
\newcommand{\A}{\mathcal A}

\newcommand{\Om}{\Omega}
\newcommand{\eps}{\varepsilon}

\newcommand{\Riesz}{\mathcal R}

\begin{document}

\begin{frontmatter}

\title{Moving One-Component Regularity Criteria for the 3D Incompressible MHD Equations}

\author[addr1]{Maotuo Guo\corref{cor1}}
\ead{18813122370@163.com}
\cortext[cor1]{Corresponding author: Maotuo Guo, Department of Mathematics, Harbin University of Science and Technology, Harbin, China. Email address: 18813122370@163.com.}
\address[addr1]{Department of Mathematics, Harbin University of Science and Technology, Harbin, China}

\begin{abstract}
We establish a scaling-critical continuation criterion for the
three-dimensional incompressible magnetohydrodynamic equations in $\R^3$ with
arbitrary positive viscosity and magnetic diffusivity.  Let $\beta(t)$ be a
unit vector that is piecewise $H^1$ in time and has only finitely many jumps.
For $3\le p<\infty$, set $\gamma_p=2p/(2p-3)$.  We prove that an $H^1$ strong
solution can be continued beyond $T_*$ if
\[
  \int_0^{T_*}\left(
  \|u(t)\cdot\beta(t)\|_{\dot H^{3/2}}^2+
  \|b(t)\cdot\beta(t)\|_{\dot H^{3/2}}^2+
  \|\beta(t)\cdot\operatorname{curl}b(t)\|_{L^p}^{\gamma_p}
  \right)\,dt<\infty .
\]
Thus the observed component may vary in time, and the magnetic assumption is
reduced to one moving magnetic component together with one Serrin-type
current-density component.  The proof is based on a moving-frame formulation of
the vorticity-current system, an anisotropic product estimate adapted to the
moving frame, and a horizontal Hodge decomposition that controls the
current-Jacobian residual.
\end{abstract}

\begin{keyword}
magnetohydrodynamics \sep regularity criterion \sep one velocity component \sep moving direction
\MSC[2020] 35Q35 \sep 35B65 \sep 76W05 \sep 76D03
\end{keyword}

\end{frontmatter}

\section{Introduction and main result}\label{sec:introduction}

We consider the Cauchy problem for the three-dimensional incompressible MHD
system
\begin{equation}\label{eq:MHD}
\left\{
\begin{aligned}
&\pa_t u+u\cdot\grad u-b\cdot\grad b-\mu\Delta u+\grad P=0,\\
&\pa_t b+u\cdot\grad b-b\cdot\grad u-\eta\Delta b=0,\\
&\Div u=\Div b=0,\\
&(u,b)|_{t=0}=(u_0,b_0),
\end{aligned}\right.
\end{equation}
where $u$ is the velocity field, $b$ is the magnetic field, and
\begin{equation}\label{eq:positivecoefficients}
  \mu>0,\qquad \eta>0.
\end{equation}
The two coefficients are allowed to be different.  We work on $\R^3$ and assume
\[
  (u_0,b_0)\in H^1(\R^3),\qquad \Div u_0=\Div b_0=0.
\]
The standard $H^1$ strong solution satisfies
\[
  (u,b)\in C([0,T);H^1)\cap L^2_{\rm loc}([0,T);\dot H^2).
\]
The basic energy inequality is
\begin{equation}\label{eq:energy}
\begin{aligned}
&\|u(t)\|_2^2+\|b(t)\|_2^2
  +2\int_0^t\bigl(\mu\|\nabla u(s)\|_2^2+
  \eta\|\nabla b(s)\|_2^2\bigr)\,ds       \\
&\qquad \le
  \|u_0\|_2^2+\|b_0\|_2^2=:E_0.
\end{aligned}
\end{equation}
We set
\begin{equation}\label{eq:lambda}
  \lambda:=\min\{\mu,\eta\}>0.
\end{equation}
All constants below may depend on $\mu$ and $\eta$ through negative powers of
$\lambda$, but no estimate uses the special relation $\mu=\eta$.

The regularity problem for three-dimensional incompressible fluids has a long
history, beginning with Leray's work and the subsequent Prodi--Serrin theory
\cite{Leray,Prodi1959,Serrin1962}.  Classical continuation and endpoint
criteria for the Navier--Stokes and Euler equations also motivate the search for
scaling-critical quantities that preclude finite-time breakdown
\cite{BealeKatoMajda,EscauriazaSereginSverak}.  For MHD, the basic weak and
strong solution theory goes back to Duvaut--Lions and Sermange--Temam
\cite{DuvautLions,SermangeTemam}; see also the regularity criteria in
\cite{HeXin2005,WuMHD,CaoWuMHD,JiLee2010,JiaZhou2012Partial,CheskidovDai2025}.
A central theme in this theory is to understand how little information is needed
to rule out singularity formation.

In the Navier--Stokes case, Chemin and Zhang proved a critical one-component
criterion \cite{CheminZhang}, and Chemin--Zhang--Zhang treated the general
one-component case \cite{CheminZhangZhang}.  Liu and Zhang recently introduced a
moving-direction criterion in which the controlled quantity is
$u(t)\cdot\be(t)$ for a time-dependent unit vector $\be(t)$
\cite{LiuZhang2024}.  Component-reduced criteria for MHD are more rigid because
the induction equation couples the velocity and magnetic fields at the level of
both vorticity and current density.  Several different forms of ``one-component''
information have therefore appeared in the MHD literature.  Criteria involving
one velocity component together with the full magnetic field were obtained in
\cite{JiaZhou2012Partial,LiuMHD}, and the critical Sobolev range in this
full-field setting was sharpened by Han--Zhao \cite{HanZhao2020}.  Criteria
involving one velocity component and several magnetic components were developed
in \cite{JiaZhou2016OneVelocity}.  Yamazaki obtained criteria involving one
velocity component and one current-density component and developed related
component-reduction arguments \cite{Yamazaki2014,YamazakiComponentReduction};
Lorentz-space versions of this velocity-current type were later proved in
\cite{AgarwalGalaRagusa2020}.  Yamazaki also studied criteria involving one
velocity component and one vorticity component \cite{Yamazaki2016Vorticity}.
Local Serrin-type criteria involving one velocity component and one magnetic
component were obtained by Chen, Qian and Zhang for suitable weak solutions in
the equal-coefficient case \cite{ChenQianZhangMHD}.  Directional-derivative
criteria, which are different in nature from component criteria, were studied in
\cite{QiaoLiu2018,QianSuZhang2025}.

The purpose of this paper is to prove a moving-direction analogue for the
three-dimensional viscous and resistive MHD equations.  The criterion forms a
Serrin-type scaling-critical family.  Fix $3\le p<\infty$ and set
\[
  \gamma_p:=\frac{2p}{2p-3}.
\]
This family is scaling invariant in the following precise sense.  If
\[
  u_\rho(t,x)=\rho u(\rho^2t,\rho x),\qquad
  b_\rho(t,x)=\rho b(\rho^2t,\rho x),
\]
and if the direction is rescaled by
\[
  \be_\rho(t)=\be(\rho^2t),
\]
then
\[
  \|u_\rho(t)\cdot\be_\rho(t)\|_{\dot H^{3/2}}^2\,dt,
  \qquad
  \|b_\rho(t)\cdot\be_\rho(t)\|_{\dot H^{3/2}}^2\,dt,
  \qquad
  \|\be_\rho(t)\cdot\curl b_\rho(t)\|_{L^p}^{\gamma_p}\,dt
\]
are invariant.  The result has three distinctive features.  First, the observed direction is
allowed to move in time: the controlled velocity component is
$u(t)\cdot\be(t)$ with $\be$ piecewise $H^1$ and with finitely many jumps.
Second, the argument is carried out in the original $(u,b)$ variables and
therefore applies to all positive viscosity--diffusivity pairs $(\mu,\eta)$,
without assuming $\mu=\eta$.  Third, the magnetic assumption does not involve
the full magnetic field.  It consists of one moving magnetic component,
$b(t)\cdot\be(t)$, and one Serrin-type current-density component,
$\be(t)\cdot\curl b(t)$.

This criterion cannot be obtained by applying a fixed-direction result after a
time-dependent rotation.  Such a rotation produces commutator terms involving
the derivative of the moving frame.  Moreover, in the MHD system the
current-density equation contains a horizontal Jacobian residual that has no
counterpart in the Navier--Stokes equations.  Controlling this residual is the
main new difficulty.  The residual contains the transverse magnetic gradient
$\nabla_{\be^\perp}b^{\be^\perp}$.  A two-dimensional Hodge decomposition in the
transverse variables splits this gradient into a horizontal-curl part and a
horizontal-divergence part:
\[
  \nabla_{\be^\perp}b^{\be^\perp}
  =\mathcal R_{\be^\perp}(\be\cdot\curl b)
  +\mathcal S_{\be^\perp}\partial_\be(b\cdot\be).
\]
The first part is estimated in $L^p$, while the second is estimated in $L^3$ by
$\|b\cdot\be\|_{\dot H^{3/2}}$.  Young's inequality then produces exactly the
critical exponent $\gamma_p=2p/(2p-3)$.  This is the mechanism that replaces the
full magnetic $\dot H^{3/2}$ norm by the magnetic component plus Serrin-type
current-density component appearing in the main criterion.  The additional
current-density component is therefore structural rather than technical.

Before stating the result, we use the following minimal terminology.  An
admissible moving direction on $[0,T)$ is a map $\be:[0,T)\to\Sph^2$ with
finitely many jump discontinuities in $(0,T)$ and with $\be'\in L^2$ on each
continuity interval.  The associated moving frame, component notation, and
horizontal Hodge representation are fixed in Section~\ref{sec:tools}.

\begin{theorem}\label{thm:main}
Let $(u,b)$ be the maximal $H^1$ strong solution of \eqref{eq:MHD} on
$[0,T_*)$ with $\mu,\eta>0$, and let $\be$ be an admissible moving direction on
$[0,T_*)$.  Fix
$3\le p<\infty$ and set
\[
  \gamma_p=\frac{2p}{2p-3},\qquad
  u^\be=u\cdot\be(t),\qquad
  b^\be=b\cdot\be(t),\qquad
  j^\be=\be(t)\cdot\curl b.
\]
If $T_*<\infty$ and
\begin{equation}\label{eq:criterionAssumption}
  \int_0^{T_*}\Bigl(
  \|u^\be(t)\|_{\dot H^{3/2}}^2
  +\|b^\be(t)\|_{\dot H^{3/2}}^2
  +\|j^\be(t)\|_{L^p}^{\gamma_p}\Bigr)\,dt<\infty,
\end{equation}
then $(u,b)$ extends as an $H^1$ strong solution beyond $T_*$.  Consequently, if
$T_*<\infty$, then for every admissible moving direction $\be$ on $[0,T_*)$ and
every $3\le p<\infty$,
\begin{equation}\label{eq:blowupCriterion}
  \int_0^{T_*}\Bigl(
  \|u^\be(t)\|_{\dot H^{3/2}}^2
  +\|b^\be(t)\|_{\dot H^{3/2}}^2
  +\|j^\be(t)\|_{L^p}^{\gamma_p}\Bigr)\,dt=\infty.
\end{equation}
\end{theorem}

\begin{corollary}[Fixed direction]\label{cor:fixed}
For any fixed unit vector $e\in\Sph^2$ and any $3\le p<\infty$, if the maximal
time $T_*$ is finite, then
\[
  \int_0^{T_*}\Bigl(
  \|u(t)\cdot e\|_{\dot H^{3/2}}^2
  +\|b(t)\cdot e\|_{\dot H^{3/2}}^2
  +\|e\cdot\curl b(t)\|_{L^p}^{\gamma_p}\Bigr)\,dt=\infty.
\]
\end{corollary}

\begin{remark}
The extra current-density component in Theorem~\ref{thm:main} is structural.  In
the moving-frame current equation, the transport terms and the normal stretching
terms cancel, but a horizontal Jacobian residual remains:
\[
  2\sum_{\ell\in\{\ta,\nv\}}J_\be(b^\ell,u^\ell),
\]
which involves $\nabla_{\be^\perp}b^{\be^\perp}$.  The transverse Hodge
decomposition controls this gradient through
$j^\be=\be\cdot\curl b$ and $\partial_\be b^\be$, and only the latter follows
from the magnetic component $b^\be$.
\end{remark}

\begin{remark}
The lower bound $p\ge3$ is used when the current part is interpolated with the
$H^1$ control, so that $r_p=2p/(p-2)\le6$.  The exclusion of the endpoint
$p=\infty$ is tied to the horizontal Calder\'on--Zygmund step in the
current-Jacobian estimate: the horizontal Riesz transforms used there are
bounded on $L^p$ for every finite $1<p<\infty$, but not on $L^\infty$ in the
strong sense.  Reaching the endpoint would therefore require a BMO-type or other
endpoint replacement, and is not pursued here.
\end{remark}

The proof of Theorem~\ref{thm:main} is organized around four estimates.  First,
we derive the moving-frame scalar system for the velocity vorticity component,
the velocity directional divergence, the current-density component, and the
magnetic directional divergence.  Second, we prove a horizontal magnetic Hodge
estimate and use it to control the current-Jacobian residual by the moving
magnetic and current components.  Third, a moving anisotropic product estimate,
adapted from the Liu--Zhang mechanism to the MHD vorticity-current variables and
the relevant zero-order multipliers, is combined with this residual estimate to
obtain a closed moving-frame vorticity-current bound.  Finally, the resulting
transverse-gradient control is used together with an anisotropic MHD $H^1$
structure to close the full $H^1$ estimate.

The remainder of the paper is organized as follows.  Section~\ref{sec:tools}
fixes the admissible moving directions, the associated moving-frame notation,
and the analytic estimates used below.  Section~\ref{sec:vorticity_current}
derives the moving vorticity-current structure.  Section~\ref{sec:nullform}
proves the horizontal current-Jacobian estimate.  Section~\ref{sec:moving_apriori}
combines it with the moving anisotropic estimates and derives the moving
horizontal-gradient bounds.  Section~\ref{sec:proof_main} first closes the full
$H^1$ estimate and then completes the proof of the main theorem.  The appendices
record the algebraic expansions and the treatment of finitely many frame jumps.

\section{Analytic tools}\label{sec:tools}

We first fix the moving-frame notation and the horizontal Hodge decomposition
used throughout the proof.  We then recall the structured iteration lemma and
establish the moving anisotropic product estimate needed below.  The latter is
modeled on the moving one-velocity-component Navier--Stokes method of
Liu--Zhang \cite{LiuZhang2024}, but is formulated here for the MHD
vorticity-current variables and includes the zero-order multipliers that occur
in the current equation.  A further estimate specific to the MHD coupling is
proved in Section~\ref{sec:nullform}.

\subsection{Moving-frame notation and horizontal Hodge decomposition}

For $T>0$, define the class of admissible directions
\begin{equation}\label{eq:Omega}
\Om(T)=\left\{\be:[0,T)\to\Sph^2:
\begin{array}{l}
\be \hbox{ has finitely many jump discontinuities in }(0,T),\\
\be'\in L^2(I)\hbox{ on each continuity interval }I
\end{array}
\right\}.
\end{equation}
As in \cite{LiuZhang2024}, for each $\be\in\Om(T)$ one can choose
$\ta,\nv\in\Om(T)$ such that, on every continuity interval,
\begin{equation}\label{eq:frame}
  \ta\cdot\nv=\nv\cdot\be=\be\cdot\ta=0,
  \qquad
  \ta\cdot(\nv\times\be)=1,
\end{equation}
and
\begin{equation}\label{eq:frameDerivative}
  |\ta'(t)|+|\nv'(t)|\le C|\be'(t)|.
\end{equation}
For a scalar function $f$ and for $\ell\in\{\ta,\nv,\be\}$ we set
\[
  \pa_\ell f:=\ell(t)\cdot\nabla f.
\]
For a vector field $v$, its components in the moving frame are denoted by
\begin{equation}\label{eq:components}
  v^\ell:=v\cdot\ell,\qquad \ell\in\{\ta,\nv,\be\},
\end{equation}
and we write
\[
  v^{\be^\perp}:=v^\ta\ta+v^\nv\nv.
\]
Furthermore,
\begin{equation}\label{eq:omegaD}
  \omega_v^\be:=\pa_{\ta} v^\nv-\pa_{\nv} v^\ta,
  \qquad
  d_v^\be:=\pa_{\be} v^\be.
\end{equation}
When $\Div v=0$,
\begin{equation}\label{eq:2DHodgeIdentities}
  \pa_{\ta} v^\ta+\pa_{\nv} v^\nv=-d_v^\be,
  \qquad
  \pa_{\ta} v^\nv-\pa_{\nv} v^\ta=\omega_v^\be.
\end{equation}
Consequently the horizontal part $v^{\be^\perp}=(v^\ta,v^\nv)$ has the
two-dimensional Hodge representation
\begin{equation}\label{eq:Hodge}
  v^{\be^\perp}=\nabla_{\be^\perp}^{\perp}\Delta_{\be^\perp}^{-1}\omega_v^\be
       -\nabla_{\be^\perp}\Delta_{\be^\perp}^{-1}d_v^\be,
\end{equation}
where
\[
  \nabla_{\be^\perp}=(\pa_{\ta},\pa_{\nv}),\qquad
  \nabla_{\be^\perp}^{\perp}=(-\pa_{\nv},\pa_{\ta}),
  \qquad
  \Delta_{\be^\perp}=\pa_{\ta}^2+\pa_{\nv}^2.
\]
Here $\Delta_{\be^\perp}^{-1}$ is the two-dimensional inverse Laplacian acting
on each transverse plane $\R^2_{\be^\perp}$, with the coordinate $x_\be$
kept fixed.  The associated horizontal Riesz transforms are therefore
fiberwise Calder\'on--Zygmund operators, and their $L^q(\R^3)$ bounds
($1<q<\infty$) are uniform in the orthonormal frame.
In particular, for $s\ge0$,
\begin{equation}\label{eq:HodgeEstimate}
  \|\nabla_{\be^\perp} v^{\be^\perp}\|_{\dot H^s}
  \lesssim
  \|\omega_v^\be\|_{\dot H^s}+\|d_v^\be\|_{\dot H^s}.
\end{equation}

\subsection{An iteration lemma}

The following form is the one needed below.  The point is that the singular
weight is not arbitrary: it has a product structure built from two fixed
$L^1_t$ quantities.

\begin{lemma}\label{lem:iteration}
Let $0<\sigma_0<1/2$.  Let $f\in W^{1,1}_{\rm loc}([0,T))$ be
nonnegative with $f(0)<\infty$.  Let $X,Y,V_1,V_2$ be nonnegative functions in
$L^1(0,T)$.  Assume that there exists a full-measure set $E\subset(0,T)$ such
that, for every $t\in E$ and every $\sigma\in(0,\sigma_0]$,
\begin{equation}\label{eq:singularInequality}
\begin{aligned}
  \frac{d}{dt}f(t)
  &\le M_2V_2(t)+M_1V_1(t)f(t)  \\
  &\quad+
  \frac{M}{\sigma}
  f(t)^{\frac1{1-\sigma}}
  X(t)^{\frac{1-2\sigma}{1-\sigma}}
  Y(t)^{\frac{\sigma}{1-\sigma}},
\end{aligned}
\end{equation}
where $M,M_1,M_2$ are independent of $\sigma$.  Then
\[
  \sup_{0\le t<T}f(t)<\infty .
\]
More quantitatively, the bound depends only on $f(0)$, $M,M_1,M_2,\sigma_0$,
and the four integrals
\[
  \int_0^T X(t)\,dt,
  \quad
  \int_0^T Y(t)\,dt,
  \quad
  \int_0^T V_1(t)\,dt,
  \quad
  \int_0^T V_2(t)\,dt .
\]
\end{lemma}

\begin{proof}
The essential point is that the singular factor $1/\sigma$ is compensated by
optimizing in $\sigma$.  We first record the elementary estimate
\begin{equation}\label{eq:optSigma}
  \inf_{0<\sigma\le\sigma_0}
  \frac1\sigma r^{\frac{\sigma}{1-\sigma}}
  \le C_{\sigma_0}\log(e+r),
  \qquad r\ge0 .
\end{equation}
Indeed, the case $r=0$ is immediate.  Assume $r>0$.  If
$\log(e+r)\ge 2/\sigma_0$, take $\sigma=1/\log(e+r)$.  Then
$0<\sigma\le\sigma_0/2$ and
\[
  \frac1\sigma r^{\frac{\sigma}{1-\sigma}}
  \le \log(e+r)\exp\left(\frac{\log r}{\log(e+r)-1}\right)
  \le C\log(e+r).
\]
If $\log(e+r)<2/\sigma_0$, taking $\sigma=\sigma_0/2$ gives a bounded
quantity, which is also dominated by $C_{\sigma_0}\log(e+r)$ because
$\log(e+r)\ge1$.  This proves \eqref{eq:optSigma}.

For each $t\in E$, inequality \eqref{eq:singularInequality} holds for all
$\sigma\in(0,\sigma_0]$, and hence we may take the infimum over $\sigma$.  If
$X(t)>0$, then
\[
\begin{aligned}
&\inf_{0<\sigma\le\sigma_0}
\frac1\sigma
f^{\frac1{1-\sigma}}
X^{\frac{1-2\sigma}{1-\sigma}}
Y^{\frac{\sigma}{1-\sigma}}     \\
&\qquad
= fX\inf_{0<\sigma\le\sigma_0}
\frac1\sigma
\left(\frac{fY}{X}\right)^{\frac{\sigma}{1-\sigma}}
\le C_{\sigma_0}fX\log\left(e+\frac{fY}{X}\right).
\end{aligned}
\]
If $X(t)=0$, the singular product vanishes because
$(1-2\sigma)/(1-\sigma)>0$; we use the convention
\[
  X\log\left(e+\frac{fY}{X}\right)=0
  \quad\hbox{on } \{X=0\}.
\]
Thus, for a.e. $t$,
\begin{equation}\label{eq:singularLogBound}
\begin{aligned}
  f'(t)
  &\le M_2V_2(t)+M_1V_1(t)f(t)   \\
  &\quad +C_{\sigma_0}M f(t)X(t)
  \log\left(e+\frac{f(t)Y(t)}{X(t)}\right).
\end{aligned}
\end{equation}
The quotient inside the logarithm is harmless.  For $X>0$,
\[
\begin{aligned}
X\log\left(e+\frac{fY}{X}\right)
&\le X\log(e+f)+X\log\left(e+\frac{Y}{X}\right) \\
&\le X\log(e+f)+C(X+Y),
\end{aligned}
\]
where the first inequality follows from $e+ab\le(e+a)(e+b)$ and the second from
$s\log(e+r/s)\le C(s+r)$ for $r,s\ge0$.  With the preceding convention on
$\{X=0\}$, \eqref{eq:singularLogBound} gives
\begin{equation}\label{eq:logGronwallBefore}
  f'(t)
  \le M_2V_2(t)
  +C\bigl(V_1(t)+X(t)+Y(t)\bigr)f(t)
  +CX(t)f(t)\log(e+f(t)),
\end{equation}
where $C$ depends only on $M,M_1,\sigma_0$.

Set
\[
  g(t)=\log(e+f(t)).
\]
Since $g\ge1$ and $0\le f/(e+f)\le1$, we obtain from \eqref{eq:logGronwallBefore}
\[
\begin{aligned}
  g'(t)
  &=\frac{f'(t)}{e+f(t)}  \\
  &\le C\bigl(V_1(t)+V_2(t)+X(t)+Y(t)\bigr)g(t)
\end{aligned}
\]
for a.e. $t\in(0,T)$, where now $C$ depends only on
$M,M_1,M_2,\sigma_0$.  Gronwall's inequality yields
\[
  g(t)
  \le g(0)\exp\left(
  C\int_0^T\bigl(V_1+V_2+X+Y\bigr)(s)\,ds
  \right),
  \qquad 0\le t<T.
\]
Hence $g$, and therefore $f$, is bounded on $[0,T)$.
\end{proof}

\subsection{Moving anisotropic estimates}

In this subsection the time variable is fixed.  Thus $\be=\be(t)$ is a fixed
unit vector, and all constants below are uniform with respect to this direction.
We use the functional spaces and dyadic operators in the same form as
Liu--Zhang \cite[Definitions~3.1--3.2]{LiuZhang2024}.  Namely,
$\R^2_{\be^\perp}$ denotes the plane orthogonal to $\be$, and $\R_{\be}$ denotes
the line parallel to $\be$.  If $X$ and $Y$ are Banach spaces on
$\R^2_{\be^\perp}$ and $\R_\be$, respectively, then
$X_{\R^2_{\be^\perp}}(Y_{\R_\be})$ denotes the corresponding mixed space.  In
particular, we write $L^p_{\be^\perp}(L^q_{\be})$ for
$L^p(\R^2_{\be^\perp};L^q(\R_\be))$.

For $s_1,s_2\in\R$, the anisotropic Sobolev space
$\dot H^{s_1}_{\be^\perp}(\dot H^{s_2}_{\be})$ is denoted briefly by
$\dot H^{s_1,s_2}_\be$ and is equipped with the norm
\begin{equation}\label{eq:anisSobolevNormLZ}
  \|f\|_{\dot H^{s_1,s_2}_\be}^2
  :=\int_{\R^3}|\xi\times\be|^{2s_1}|\xi\cdot\be|^{2s_2}
  |\widehat f(\xi)|^2\,d\xi .
\end{equation}
Equivalently, if $(\ta,\nv,\be)$ is an orthonormal frame satisfying
\eqref{eq:frame}, then
$|\xi\times\be|^2=|\xi\cdot\ta|^2+|\xi\cdot\nv|^2$.

We also use anisotropic homogeneous dyadic blocks adapted to the decomposition
\(\mathbb R^3=\mathbb R^2_{\be^\perp}\times \mathbb R_{\be}\).  Let
$\varphi,\chi\in C^\infty([0,\infty))$ be such that
\[
  \operatorname{supp}\varphi\subset\left\{r:\frac34\le r\le\frac83\right\},
  \qquad \sum_{j\in\mathbb Z}\varphi(2^{-j}r)=1\quad(r>0),
\]
\[
  \operatorname{supp}\chi\subset\left\{r:0\le r\le\frac43\right\},
  \qquad \chi(r)+\sum_{j\ge0}\varphi(2^{-j}r)=1\quad(r\ge0).
\]
For a tempered distribution $f$ on $\R^3$, set
\begin{equation}\label{eq:LZdyadicBlocks}
\begin{aligned}
  \Delta_k^{\be^\perp}f
  &:=\mathcal F^{-1}\!\left(\varphi(2^{-k}|\xi\times\be|)\widehat f(\xi)\right),
  &
  S_k^{\be^\perp}f
  &:=\mathcal F^{-1}\!\left(\chi(2^{-k}|\xi\times\be|)\widehat f(\xi)\right),\\
  \Delta_\ell^{\be}f
  &:=\mathcal F^{-1}\!\left(\varphi(2^{-\ell}|\xi\cdot\be|)\widehat f(\xi)\right),
  &
  S_\ell^{\be}f
  &:=\mathcal F^{-1}\!\left(\chi(2^{-\ell}|\xi\cdot\be|)\widehat f(\xi)\right).
\end{aligned}
\end{equation}
Let $p,q_1,q_2\in[1,\infty]$ and $s_1,s_2\in\R$.  The anisotropic Besov norm is
\begin{equation}\label{eq:anisBesovGeneral}
\begin{aligned}
  \|f\|_{(\dot B^{s_1}_{p,q_1})_{\be^\perp}(\dot B^{s_2}_{p,q_2})_{\be}}
  :=
  \left\|
  \left(
  2^{ks_1}
  \left\|
  \left(2^{\ell s_2}
  \|\Delta_k^{\be^\perp}\Delta_\ell^\be f\|_{L^p(\R^3)}
  \right)_{\ell\in\mathbb Z}
  \right\|_{\ell^{q_2}(\mathbb Z)}
  \right)_{k\in\mathbb Z}
  \right\|_{\ell^{q_1}(\mathbb Z)} .
\end{aligned}
\end{equation}
The homogeneous spaces are understood modulo polynomials in the standard sense
of homogeneous Littlewood--Paley theory; see, for instance,
\cite{BahouriCheminDanchin}.  Equivalently, the estimates below are first proved
for Schwartz functions whose Fourier transforms vanish near the origin, and the
general case follows by the usual density and regularization argument.  When a
low-frequency condition is needed, we impose
$\lim_{j\to-\infty}\|(S_j^{\be^\perp}f,S_j^\be f)\|_{L^\infty}=0$.  The order of
summation in \eqref{eq:anisBesovGeneral}, namely the $\be$-summation inside the
$\be^\perp$-summation, will be used below in the duality step.  In particular,
\begin{equation}\label{eq:anisBesovNorm}
  \|f\|_{(\dot B^{s}_{2,2})_{\be^\perp}(\dot B^{r}_{2,1})_{\be}}
  :=\left(\sum_{k\in\mathbb Z}2^{2ks}
  \left(\sum_{\ell\in\mathbb Z}2^{\ell r}
  \|\Delta_k^{\be^\perp}\Delta_\ell^\be f\|_2\right)^2\right)^{1/2},
\end{equation}
while
\begin{equation}\label{eq:anisBesovNormInf}
  \|f\|_{(\dot B^{s}_{2,2})_{\be^\perp}(\dot B^{r}_{2,\infty})_{\be}}
  :=\left(\sum_{k\in\mathbb Z}2^{2ks}
  \left(\sup_{\ell\in\mathbb Z}2^{\ell r}
  \|\Delta_k^{\be^\perp}\Delta_\ell^\be f\|_2\right)^2\right)^{1/2}.
\end{equation}
The Littlewood--Paley characterization gives
\begin{equation}\label{eq:anisSobolevLP}
  \|f\|_{\dot H^{s_1,s_2}_\be}
  \simeq
  \|f\|_{(\dot B^{s_1}_{2,2})_{\be^\perp}(\dot B^{s_2}_{2,2})_\be},
\end{equation}
and the constants are independent of the fixed unit vector $\be$.

The following two-dimensional product estimate is used in the horizontal
variables.  It is invariant under rotations and will be applied on the
transverse plane $\R^2_{\be^\perp}$.  We state it separately to make explicit
the only place where the small parameter produces an endpoint summation loss.

\begin{lemma}\label{lem:horizontalProduct}
Let $0<\sigma\le1/5$.  For Schwartz functions $a,b$ on $\R^2_{\be^\perp}$,
\begin{equation}\label{eq:horizontalProduct2D}
  \|ab\|_{\dot H^{\sigma}(\R^2_{\be^\perp})}
  \le \frac{C}{\sqrt{\sigma}}
  \|a\|_{\dot H^{1-\sigma}(\R^2_{\be^\perp})}
  \|b\|_{\dot H^{2\sigma}(\R^2_{\be^\perp})} .
\end{equation}
The constant $C$ is independent of $\sigma\in(0,1/5]$ and of the unit
direction $\be$.
\end{lemma}

\begin{proof}
We use the homogeneous dyadic blocks in the plane $\R^2_{\be^\perp}$.  Let
\[
\begin{aligned}
  \widetilde\Delta_k^{\be^\perp}
  &:=\Delta_{k-1}^{\be^\perp}+\Delta_k^{\be^\perp}
    +\Delta_{k+1}^{\be^\perp},\qquad
  A_k:=2^{(1-\sigma)k}\|\Delta_k^{\be^\perp} a\|_{L^2(\R^2_{\be^\perp})},\\
  B_k&:=2^{2\sigma k}\|\Delta_k^{\be^\perp} b\|_{L^2(\R^2_{\be^\perp})}.
\end{aligned}
\]
Then $\|A\|_{\ell^2}\simeq\|a\|_{\dot H^{1-\sigma}(\R^2_{\be^\perp})}$ and
$\|B\|_{\ell^2}\simeq\|b\|_{\dot H^{2\sigma}(\R^2_{\be^\perp})}$.  Bony's
horizontal decomposition reads
\[
  ab=T_a^{\be^\perp}b+T_b^{\be^\perp}a+R^{\be^\perp}(a,b),
\]
where
\[
  T_a^{\be^\perp}b
  =\sum_jS_{j-1}^{\be^\perp}a\,\Delta_j^{\be^\perp}b,
  \qquad
  T_b^{\be^\perp}a
  =\sum_jS_{j-1}^{\be^\perp}b\,\Delta_j^{\be^\perp}a,
  \qquad
  R^{\be^\perp}(a,b)
  =\sum_j\Delta_j^{\be^\perp}a\,\widetilde\Delta_j^{\be^\perp}b .
\]

For the low--high term, the two-dimensional Bernstein inequality gives
\[
  \|S_{j-1}^{\be^\perp}a\|_{L^\infty(\R^2_{\be^\perp})}
  \lesssim \sum_{j'<j}2^{j'}
  \|\Delta_{j'}^{\be^\perp}a\|_2 .
\]
Hence, up to harmless finite overlaps of dyadic annuli,
\[
\begin{aligned}
  2^{\sigma k}\|\Delta_k^{\be^\perp}T_a^{\be^\perp}b\|_2
  &\lesssim \sum_{|j-k|\le C_0}
  2^{\sigma j}\|S_{j-1}^{\be^\perp}a\|_\infty
  \|\Delta_j^{\be^\perp}b\|_2  \\
  &\lesssim \sum_{|j-k|\le C_0}
  B_j\sum_{j'<j}2^{-\sigma(j-j')}A_{j'} .
\end{aligned}
\]
If $K_m=2^{-\sigma m}{\bf 1}_{m\ge1}$, then
$\|K\|_{\ell^2}\lesssim\sigma^{-1/2}$.  Therefore
\[
  \|T_a^{\be^\perp}b\|_{\dot H^\sigma(\R^2_{\be^\perp})}
  \lesssim \|B\|_{\ell^2}\|K*A\|_{\ell^\infty}
  \lesssim \frac{1}{\sqrt\sigma}\|A\|_{\ell^2}\|B\|_{\ell^2} .
\]

For the high--low term, similarly,
\[
\begin{aligned}
  2^{\sigma k}\|\Delta_k^{\be^\perp}T_b^{\be^\perp}a\|_2
  &\lesssim \sum_{|j-k|\le C_0}
  A_j\sum_{j'<j}2^{-(1-2\sigma)(j-j')}B_{j'} .
\end{aligned}
\]
Since $0<\sigma\le1/5$, the kernel
$2^{-(1-2\sigma)m}{\bf 1}_{m\ge1}$ has an $\ell^2$ norm bounded by an absolute
constant.  Thus
\[
  \|T_b^{\be^\perp}a\|_{\dot H^\sigma(\R^2_{\be^\perp})}
  \lesssim \|A\|_{\ell^2}\|B\|_{\ell^2} .
\]

For the remainder term, choose exponents
$(2/\sigma,2/(1-\sigma))$ so that
$\sigma/2+(1-\sigma)/2=1/2$.  The two-dimensional Bernstein inequality yields
\[
  \|\Delta_j^{\be^\perp}a\|_{L^{2/\sigma}}
  \lesssim 2^{(1-\sigma)j}\|\Delta_j^{\be^\perp}a\|_2=A_j,
\]
and
\[
  \|\widetilde\Delta_j^{\be^\perp}b\|_{L^{2/(1-\sigma)}}
  \lesssim 2^{\sigma j}\|\widetilde\Delta_j^{\be^\perp}b\|_2
  \lesssim 2^{-\sigma j}\widetilde B_j,
\]
where $\widetilde B_j:=B_{j-1}+B_j+B_{j+1}$.  Consequently
\[
  \|\Delta_j^{\be^\perp}a\,\widetilde\Delta_j^{\be^\perp}b\|_2
  \lesssim 2^{-\sigma j}A_j\widetilde B_j .
\]
The block $\Delta_k^{\be^\perp}R^{\be^\perp}(a,b)$ receives contributions only
from $j\ge k-C_0$, and hence
\[
  2^{\sigma k}\|\Delta_k^{\be^\perp}R^{\be^\perp}(a,b)\|_2
  \lesssim\sum_{j\ge k-C_0}2^{-\sigma(j-k)}A_j\widetilde B_j .
\]
Using Young's inequality in the form
$\ell^2*\ell^1\to\ell^2$ and Cauchy's inequality,
\[
\begin{aligned}
  \|R^{\be^\perp}(a,b)\|_{\dot H^\sigma(\R^2_{\be^\perp})}
  &\lesssim \left(\sum_{m\ge-C_0}2^{-2\sigma m}\right)^{1/2}
  \|A\widetilde B\|_{\ell^1}  \\
  &\lesssim \frac{1}{\sqrt\sigma}\|A\|_{\ell^2}\|B\|_{\ell^2} .
\end{aligned}
\]
Combining the three estimates proves \eqref{eq:horizontalProduct2D}.  The
argument uses only Bernstein inequalities and Bony's decomposition on the plane
$\R^2_{\be^\perp}$, so the constant is invariant under rotations of $\be$.
\end{proof}

The next lemma is the moving anisotropic product estimate needed in the MHD
argument.

\begin{lemma}\label{lem:LZproduct}
Let $0<\sigma\le1/5$.  Let $q$ be one of the quantities
$\omega_u^\be,d_u^\be,\omega_b^\be,d_b^\be$, let $V,Z\in\{u,b\}$, and let
$\nabla_{\be^\perp}^\#$ denote either $\nabla_{\be^\perp}$ or
$\nabla_{\be^\perp}^{\perp}$.  Let $\mathcal R$ be a degree-zero Mikhlin
multiplier whose symbol satisfies the standard Mikhlin bounds uniformly in the
moving frame.  In particular, the estimate applies to finite compositions of
Riesz transforms and to the operators $\partial_i\partial_j\Delta^{-1}$
occurring below.  Then
\begin{equation}\label{eq:LZmixed}
\begin{aligned}
&\left|\int_{\R^3}
  \bigl(\pa_{\be} V^{\be^\perp}\cdot\nabla_{\be^\perp}^\# Z^\be\bigr)\,\mathcal Rq\,dx\right|        \\
&\quad \le
  \eps\D
  +\frac{C_{\eps}}{\sigma}
  \F^{\frac1{1-\sigma}}
  \|Z^\be\|_{\dot H^{3/2}}^{\frac{2(1-2\sigma)}{1-\sigma}}
  \A^{\frac{\sigma}{1-\sigma}} .
\end{aligned}
\end{equation}
For this lemma only, the unweighted quantities are
\begin{equation}\label{eq:FDAprelim}
\begin{aligned}
\F&=\|\omega_u^\be\|_2^2+\|d_u^\be\|_2^2
    +\|\omega_b^\be\|_2^2+\|d_b^\be\|_2^2,\\
\D&=\|\nabla\omega_u^\be\|_2^2+\|\nabla d_u^\be\|_2^2
    +\|\nabla\omega_b^\be\|_2^2+\|\nabla d_b^\be\|_2^2,\\
\A&=\|\nabla u\|_2^2+\|\nabla b\|_2^2.
\end{aligned}
\end{equation}
After the shorthand \eqref{eq:omegaCurrentNotation} is introduced, this is the
same $\F,\D,\A$ notation used in the vorticity-current estimates.
\end{lemma}

\begin{proof}
We first record the following trilinear product law.  For scalar functions
$F,G,H$ and $0<\sigma\le1/5$,
\begin{equation}\label{eq:anisProductLaw}
  \left|\int_{\R^3}F G H\,dx\right|
  \le \frac{C}{\sqrt{\sigma}}
  \|F\|_{\dot H^{1-\sigma,0}_\be}
  \|G\|_{(\dot B^{-\sigma}_{2,2})_{\be^\perp}(\dot B^{1/2}_{2,1})_{\be}}
  \|H\|_{\dot H^{2\sigma,0}_\be} .
\end{equation}
Indeed, by the duality associated with the order of summation in
\eqref{eq:anisBesovGeneral}, it is enough to prove
\begin{equation}\label{eq:productIntoDual}
  \|FH\|_{(\dot B^{\sigma}_{2,2})_{\be^\perp}(\dot B^{-1/2}_{2,\infty})_{\be}}
  \le \frac{C}{\sqrt{\sigma}}
  \|F\|_{\dot H^{1-\sigma,0}_\be}
  \|H\|_{\dot H^{2\sigma,0}_\be} .
\end{equation}
To see this directly, decompose $x=x_{\be^\perp}+x_\be\be$, with
$x_{\be^\perp}\in\R^2_{\be^\perp}$ and $x_\be\in\R$.  For each fixed $x_\be$,
Lemma~\ref{lem:horizontalProduct} gives
\[
  \|F(\cdot,x_\be)H(\cdot,x_\be)\|_{\dot H^\sigma(\R^2_{\be^\perp})}
  \le \frac{C}{\sqrt\sigma}
  \|F(\cdot,x_\be)\|_{\dot H^{1-\sigma}(\R^2_{\be^\perp})}
  \|H(\cdot,x_\be)\|_{\dot H^{2\sigma}(\R^2_{\be^\perp})} .
\]
After integration in $x_\be$ and Cauchy's inequality,
\begin{equation}\label{eq:L1verticalProduct}
  \|FH\|_{L^1_{\be}(\dot H^\sigma_{\be^\perp})}
  \le \frac{C}{\sqrt\sigma}
  \|F\|_{L^2_{\be}(\dot H^{1-\sigma}_{\be^\perp})}
  \|H\|_{L^2_{\be}(\dot H^{2\sigma}_{\be^\perp})} .
\end{equation}
The one-dimensional Bernstein inequality in the $\be$ direction implies
\[
  \sup_{\ell\in\mathbb Z}2^{-\ell/2}
  \|\Delta_\ell^\be\Phi\|_{L^2_{\be}}
  \lesssim \|\Phi\|_{L^1_{\be}} .
\]
Therefore, using \eqref{eq:anisBesovNormInf},
\begin{equation}\label{eq:L1ToBesovDual}
\begin{aligned}
\|\Phi\|_{(\dot B^{\sigma}_{2,2})_{\be^\perp}(\dot B^{-1/2}_{2,\infty})_{\be}}^2
&= \sum_k2^{2\sigma k}
   \left(\sup_\ell 2^{-\ell/2}
   \|\Delta_k^{\be^\perp}\Delta_\ell^\be\Phi\|_2\right)^2      \\
&\lesssim \sum_k2^{2\sigma k}
   \|\Delta_k^{\be^\perp}\Phi\|_{L^1_{\be}L^2_{\be^\perp}}^2              \\
&\lesssim \|\Phi\|_{L^1_{\be}(\dot H^\sigma_{\be^\perp})}^2 .
\end{aligned}
\end{equation}
Taking $\Phi=FH$ and combining \eqref{eq:L1verticalProduct} with
\eqref{eq:L1ToBesovDual} proves \eqref{eq:productIntoDual}, and hence
\eqref{eq:anisProductLaw}.

We next estimate the three factors in \eqref{eq:anisProductLaw}.  First, since
$\nabla_{\be^\perp}^\#$ is a first-order derivative in the transverse plane,
\[
  \|\nabla_{\be^\perp}^\#Z^\be\|_{(\dot B^{-\sigma}_{2,2})_{\be^\perp}(\dot B^{1/2}_{2,1})_{\be}}
  \lesssim
  \|Z^\be\|_{(\dot B^{1-\sigma}_{2,2})_{\be^\perp}(\dot B^{1/2}_{2,1})_{\be}} .
\]
The Liu--Zhang interpolation estimate \cite[Lemma~3.4]{LiuZhang2024} with
$\eta=0$ gives
\begin{equation}\label{eq:wcomponentBesovDetailed}
  \|\nabla_{\be^\perp}^\# Z^\be\|_{(\dot B^{-\sigma}_{2,2})_{\be^\perp}(\dot B^{1/2}_{2,1})_{\be}}
  \lesssim
  \|\partial_{\be} Z^\be\|_2^{2\sigma}
  \|Z^\be\|_{\dot H^{3/2}}^{1-2\sigma} .
\end{equation}
We recall the dyadic verification in the present notation.
Let
\[
  a_{k\ell}:=\|\Delta_k^{\be^\perp}\Delta_\ell^\be Z^\be\|_2,
  \qquad
  X_{k\ell}:=2^\ell a_{k\ell},
  \qquad
  Y_{k\ell}:=(2^{2k}+2^{2\ell})^{3/4}a_{k\ell}.
\]
Then
$\|X\|_{\ell^2_{k,\ell}}\simeq\|\partial_\be Z^\be\|_2$ and
$\|Y\|_{\ell^2_{k,\ell}}\simeq\|Z^\be\|_{\dot H^{3/2}}$.  Put
$\alpha=1/2-2\sigma$.  Since $0<\sigma\le1/5$, $\alpha\ge1/10$.  A direct
comparison of dyadic powers gives
\begin{equation}\label{eq:blockInterpolationW}
  2^{(1-\sigma)k}2^{\ell/2}a_{k\ell}
  \lesssim
  2^{-\alpha|k-\ell|}X_{k\ell}^{2\sigma}Y_{k\ell}^{1-2\sigma} .
\end{equation}
Indeed, for $k\ge\ell$ the powers match exactly; for $\ell\ge k$ the right-hand
side has the additional factor $2^{(1/2+\sigma)(\ell-k)}$.  With
$K_{k\ell}:=2^{-\alpha|k-\ell|}$, \eqref{eq:blockInterpolationW} gives
\[
\begin{aligned}
&\|Z^\be\|_{(\dot B^{1-\sigma}_{2,2})_{\be^\perp}(\dot B^{1/2}_{2,1})_{\be}}^2  \\
&\quad\lesssim
  \sum_k\left(\sum_\ell K_{k\ell}X_{k\ell}^{2\sigma}Y_{k\ell}^{1-2\sigma}\right)^2 .
\end{aligned}
\]
For each fixed $k$, H\"older's inequality with exponents
$1/\sigma$, $2/(1-2\sigma)$, and $2$ yields
\[
  \sum_\ell K_{k\ell}X_{k\ell}^{2\sigma}Y_{k\ell}^{1-2\sigma}
  \lesssim
  \left(\sum_\ell K_{k\ell}X_{k\ell}^2\right)^\sigma
  \left(\sum_\ell K_{k\ell}Y_{k\ell}^2\right)^{(1-2\sigma)/2},
\]
because $\sup_k\sum_\ell K_{k\ell}\lesssim1$.  Taking the outer $\ell^2_k$
norm and applying H\"older's inequality in $k$, using also
$\sup_\ell\sum_kK_{k\ell}\lesssim1$, gives
\[
\begin{aligned}
  \|Z^\be\|_{(\dot B^{1-\sigma}_{2,2})_{\be^\perp}(\dot B^{1/2}_{2,1})_{\be}}^2
  &\lesssim
  \left(\sum_{k,\ell}X_{k\ell}^2\right)^{2\sigma}
  \left(\sum_{k,\ell}Y_{k\ell}^2\right)^{1-2\sigma}.
\end{aligned}
\]
This proves \eqref{eq:wcomponentBesovDetailed}.  In the present application,
$Z=u$ or $Z=b$, and therefore
\begin{equation}\label{eq:WbetaDerivativeControl}
  \|\partial_\be Z^\be\|_2
  =\begin{cases}
    \|d_u^\be\|_2,& Z=u,\\
    \|d_b^\be\|_2,& Z=b,
  \end{cases}
  \le \F^{1/2}.
\end{equation}
Thus
\begin{equation}\label{eq:GfactorBound}
  \|\nabla_{\be^\perp}^\#Z^\be\|_{(\dot B^{-\sigma}_{2,2})_{\be^\perp}(\dot B^{1/2}_{2,1})_{\be}}
  \lesssim \F^\sigma\|Z^\be\|_{\dot H^{3/2}}^{1-2\sigma} .
\end{equation}

Second, the vertical derivative of the transverse part is controlled by the
horizontal vorticity-divergence pair.  We claim that
\begin{equation}\label{eq:partial3vhDetailed}
  \|\partial_{\be} V^{\be^\perp}\|_{\dot H^{1-\sigma,0}_\be}
  \lesssim
  \|\nabla V\|_2^\sigma
  \bigl(\|\nabla\omega_V^\be\|_2+\|\nabla d_V^\be\|_2\bigr)^{1-\sigma}
  \lesssim \A^{\sigma/2}\D^{(1-\sigma)/2} .
\end{equation}
Because $\Div V=0$, the two-dimensional Hodge identities in the plane
$\R^2_{\be^\perp}$ give, in Fourier variables,
\begin{equation}\label{eq:HodgeFourierBeta}
  |\widehat{\omega_V^\be}(\xi)|^2+|\widehat{d_V^\be}(\xi)|^2
  =|\xi\times\be|^2|\widehat{V^{\be^\perp}}(\xi)|^2 .
\end{equation}
Consequently,
\[
\begin{aligned}
\|\partial_{\be} V^{\be^\perp}\|_{\dot H^{1-\sigma,0}_\be}^2
&=\int |\xi\times\be|^{2(1-\sigma)}|\xi\cdot\be|^2
  |\widehat{V^{\be^\perp}}(\xi)|^2\,d\xi       \\
&\le \int
  \bigl(|\xi|^2|\widehat V(\xi)|^2\bigr)^\sigma
  \bigl(|\xi|^2(|\widehat{\omega_V^\be}|^2+|\widehat{d_V^\be}|^2)\bigr)^{1-\sigma}
  d\xi.
\end{aligned}
\]
H\"older's inequality proves \eqref{eq:partial3vhDetailed}.

Third, for $q\in\{\omega_u^\be,d_u^\be,\omega_b^\be,d_b^\be\}$, let $m(\xi)$ be
the symbol of $\mathcal R$.  Since $\mathcal R$ and
$(-\Delta_{\be^\perp})^\sigma$ are Fourier multipliers, and since $m$ is a
uniformly bounded degree-zero Mikhlin symbol,
\[
  \|\mathcal R q\|_{\dot H^{2\sigma,0}_\be}
  =\bigl\||\xi\times\be|^{2\sigma}m(\xi)\widehat q(\xi)\bigr\|_{L^2_\xi}
  \lesssim \|q\|_{\dot H^{2\sigma,0}_\be}.
\]
Moreover, $|\xi\times\be|\le|\xi|$, and the Fourier-side interpolation between
$L^2$ and $\dot H^1$ gives
\[
\begin{aligned}
  \|q\|_{\dot H^{2\sigma,0}_\be}^2
  &\le \int_{\R^3}|\xi|^{4\sigma}|\widehat q(\xi)|^2\,d\xi     \\
  &= \int_{\R^3}
  \bigl(|\widehat q(\xi)|^2\bigr)^{1-2\sigma}
  \bigl(|\xi|^2|\widehat q(\xi)|^2\bigr)^{2\sigma}\,d\xi       \\
  &\le \|q\|_2^{2(1-2\sigma)}\|\nabla q\|_2^{4\sigma}.
\end{aligned}
\]
Hence
\begin{equation}\label{eq:qHorizontalInterpolationDetailed}
  \|\mathcal R q\|_{\dot H^{2\sigma,0}_\be}
  \lesssim \|q\|_{\dot H^{2\sigma,0}_\be}
  \lesssim \|q\|_2^{1-2\sigma}\|\nabla q\|_2^{2\sigma}
  \lesssim \F^{(1-2\sigma)/2}\D^\sigma .
\end{equation}

Applying \eqref{eq:anisProductLaw} with
\[
  F=\partial_{\be}V^{\be^\perp},\qquad
  G=\nabla_{\be^\perp}^\#Z^\be,
  \qquad
  H=\mathcal Rq,
\]
and using \eqref{eq:GfactorBound}, \eqref{eq:partial3vhDetailed}, and
\eqref{eq:qHorizontalInterpolationDetailed}, we obtain
\begin{equation}\label{eq:preYoungLZDetailed}
\begin{aligned}
\left|\int
  (\partial_{\be} V^{\be^\perp}\cdot\nabla_{\be^\perp}^\#Z^\be)\,\mathcal Rq\,dx\right|
&\le \frac{C}{\sqrt{\sigma}}
  \F^{1/2}\|Z^\be\|_{\dot H^{3/2}}^{1-2\sigma}
  \A^{\sigma/2}\D^{(1+\sigma)/2} .
\end{aligned}
\end{equation}
Set
\[
  B:=\F^{1/2}\|Z^\be\|_{\dot H^{3/2}}^{1-2\sigma}\A^{\sigma/2}.
\]
Young's inequality with conjugate exponents $2/(1+\sigma)$ and $2/(1-\sigma)$
gives
\[
  \frac{C}{\sqrt\sigma}B\D^{(1+\sigma)/2}
  \le \eps\D
  +C_\eps\sigma^{-1/(1-\sigma)}B^{2/(1-\sigma)}.
\]
Since $0<\sigma\le1/5$, the factor
$\sigma^{-1/(1-\sigma)}$ is bounded by $C\sigma^{-1}$.  This proves
\eqref{eq:LZmixed}.
\end{proof}

\section{Moving vorticity-current system}\label{sec:vorticity_current}

Throughout this section we work on one continuity interval of $(\ta,\nv,\be)$.
The estimates are uniform on each interval and are then summed over the finitely
many intervals.  We also fix $3\le p<\infty$ and write
$\gamma_p=2p/(2p-3)$.

Set
\begin{equation}\label{eq:omegaCurrentNotation}
  \omega:=\omega_u^\be,
  \qquad d:=d_u^\be=\pa_{\be} u^\be,
  \qquad j:=\omega_b^\be,
  \qquad h:=d_b^\be=\pa_{\be} b^\be.
\end{equation}
We also set
\begin{equation}\label{eq:globalQuantities}
\begin{aligned}
\F(t)&=\|\omega(t)\|_2^2+\|d(t)\|_2^2+
       \|j(t)\|_2^2+\|h(t)\|_2^2,\\
\D_{\mu,\eta}(t)&=
\mu\bigl(\|\nabla\omega\|_2^2+\|\nabla d\|_2^2\bigr)
+\eta\bigl(\|\nabla j\|_2^2+\|\nabla h\|_2^2\bigr),\\
\D(t)&=\|\nabla\omega\|_2^2+\|\nabla d\|_2^2+
       \|\nabla j\|_2^2+\|\nabla h\|_2^2,\\
\A(t)&=\|\nabla u(t)\|_2^2+\|\nabla b(t)\|_2^2.
\end{aligned}
\end{equation}
Thus
\begin{equation}\label{eq:Dlambda}
  \D_{\mu,\eta}\ge\lambda\D.
\end{equation}
The vector vorticity and current are denoted by
\[
  \Omega=\curl u,
  \qquad
  J=\curl b.
\]
The MHD vorticity-current equations are
\begin{equation}\label{eq:vectorVorticityCurrent}
\left\{
\begin{aligned}
&\pa_t\Omega+u\cdot\nabla\Omega-b\cdot\nabla J-
\mu\Delta\Omega
=\Omega\cdot\nabla u-J\cdot\nabla b,\\
&\pa_tJ+u\cdot\nabla J-b\cdot\nabla\Omega-
\eta\Delta J
=J\cdot\nabla u-\Omega\cdot\nabla b
  +2\sum_{\ell\in\{\ta,\nv,\be\}}\nabla b^\ell\times\nabla u^\ell .
\end{aligned}\right.
\end{equation}
The last sum is the usual full component sum, written in the moving
orthonormal frame:
\[
  \sum_{\ell\in\{\ta,\nv,\be\}}\nabla b^\ell\times\nabla u^\ell
  =\sum_{i=1}^3\nabla b_i\times\nabla u_i .
\]
Projecting the vorticity-current system on $\be(t)$ and differentiating the
$\be$-components of the velocity and magnetic equations gives the following
scalar system.  The statement is exact on every continuity interval of the
moving frame; the terms denoted by $\mathcal C$ are precisely the linear
commutators produced by differentiating the time-dependent frame.

\begin{lemma}\label{lem:movingScalarSystem}
On every continuity interval of $(\ta,\nv,\be)$, the four quantities
$\omega,d,j,h$ defined in \eqref{eq:omegaCurrentNotation} satisfy
\begin{equation}\label{eq:structure}
\left\{
\begin{aligned}
&\pa_t\omega+u\cdot\nabla\omega-b\cdot\nabla j-
\mu\Delta\omega
=\mathfrak R_\omega^{NS}+\mathfrak R_\omega^{MHD}+\mathcal C_\omega,\\
&\pa_t d+u\cdot\nabla d-b\cdot\nabla h-
\mu\Delta d
=\mathfrak R_d^{NS}+\mathfrak R_d^{MHD}+\mathcal C_d,\\
&\pa_tj+u\cdot\nabla j-b\cdot\nabla\omega-
\eta\Delta j
=\mathfrak R_j^{MHD}+\mathcal N_\be+\mathcal C_j,\\
&\pa_th+u\cdot\nabla h-b\cdot\nabla d-
\eta\Delta h
=\mathfrak R_h^{MHD}+\mathcal C_h.
\end{aligned}\right.
\end{equation}
Here
\begin{equation}\label{eq:exactResidualsOmega}
\begin{aligned}
\mathfrak R_\omega^{NS}
&=d\omega-\partial_\be u^{\be^\perp}\cdot
  \nabla_{\be^\perp}^{\perp}u^\be,\\
\mathfrak R_\omega^{MHD}
&=-hj+\partial_\be b^{\be^\perp}\cdot
  \nabla_{\be^\perp}^{\perp}b^\be,
\end{aligned}
\end{equation}
\begin{equation}\label{eq:exactResidualsD}
\begin{aligned}
\mathfrak R_d^{NS}
&=-\partial_\be u\cdot\nabla u^\be
  -\partial_\be^2(-\Delta)^{-1}Q_u,\\
\mathfrak R_d^{MHD}
&=\partial_\be b\cdot\nabla b^\be
  +\partial_\be^2(-\Delta)^{-1}Q_b,
\end{aligned}
\end{equation}
where
\begin{equation}\label{eq:QuQb}
  Q_u=\sum_{\ell,m\in\{\ta,\nv,\be\}}\partial_\ell u^m\partial_m u^\ell,
  \qquad
  Q_b=\sum_{\ell,m\in\{\ta,\nv,\be\}}\partial_\ell b^m\partial_m b^\ell .
\end{equation}
The current-density and magnetic-divergence residuals are
\begin{equation}\label{eq:exactResidualsJh}
\begin{aligned}
\mathfrak R_j^{MHD}
&=d j-h\omega
+\partial_\be u^{\be^\perp}\cdot\nabla_{\be^\perp}^{\perp} b^\be
-\partial_\be b^{\be^\perp}\cdot\nabla_{\be^\perp}^{\perp} u^\be,\\
\mathfrak R_h^{MHD}
&=-\partial_\be u^{\be^\perp}\cdot\nabla_{\be^\perp} b^\be
  +\partial_\be b^{\be^\perp}\cdot\nabla_{\be^\perp} u^\be,
\end{aligned}
\end{equation}
and the purely horizontal current-Jacobian residual is
\begin{equation}\label{eq:currentNull}
  \mathcal N_\be
  =2\sum_{\ell\in\{\ta,\nv\}}J_\be(b^\ell,u^\ell),
  \qquad
  J_\be(f,g):=\pa_{\ta} f\,\pa_{\nv} g-\pa_{\nv} f\,\pa_{\ta} g.
\end{equation}
The pressure is normalized by
\begin{equation}\label{eq:pressure}
  -\Delta P=Q_u-Q_b,
\end{equation}
and $(-\Delta)^{-1}$ is understood with this convention.
\end{lemma}

\begin{proof}
Since the frame is orthonormal and depends only on time, the spatial algebra is
rotation invariant and may be verified after freezing the frame.  The
frozen-frame computations are recorded in Appendix~\ref{app:expansions}.  Differentiating
the frame in time adds only the linear commutators
$\mathcal C_\omega,\mathcal C_d,\mathcal C_j,\mathcal C_h$, which are estimated
in Lemma~\ref{lem:commutators}.  This gives the exact system above on each
continuity interval.
\end{proof}

\begin{lemma}\label{lem:commutators}
For every $\varepsilon>0$, on every continuity interval of the frame,
\begin{equation}\label{eq:commutatorLemmaBound}
\begin{aligned}
&|\langle \mathcal C_\omega,\omega\rangle|+
  |\langle \mathcal C_d,d\rangle|+
  |\langle \mathcal C_j,j\rangle|+
  |\langle \mathcal C_h,h\rangle|        \\
&\qquad\le \varepsilon\D_{\mu,\eta}
  +C_{\varepsilon,\lambda} E_0\bigl(|\ta'(t)|^2+|\nv'(t)|^2+|\be'(t)|^2\bigr).
\end{aligned}
\end{equation}
\end{lemma}

\begin{proof}
The point is that the frame depends on time but not on space.  For a divergence-free
field $z$,
\[
  \omega_z^\be=\be\cdot\curl z,
  \qquad
  d_z^\be=\be_i\be_j\partial_i z^j.
\]
Differentiating in time gives the identities
\[
  \partial_t\omega_z^\be
  =\be\cdot\curl\partial_tz+\be'\cdot\curl z,
\]
and
\[
  \partial_td_z^\be
  =\be_i\be_j\partial_i\partial_tz^j
  +(\be_i'\be_j+\be_i\be_j')\partial_i z^j .
\]
Hence the extra terms created by differentiating the frame are linear
combinations of
\begin{equation}\label{eq:commutatorForms}
  \be'\cdot\curl z,
  \qquad
  (\be_i'\be_j+\be_i\be_j')\partial_i z^j,
\end{equation}
with $z=u$ or $z=b$.  If the frame is written through
$(\tau,\nu,\be)$ rather than only through $\be$, the same list is obtained after
using $\be=\tau\times\nu$; in any case the coefficients are bounded by
$C(|\tau'|+|\nu'|+|\be'|)$.

We show the estimate for the two displayed model terms.  Since the coefficients
in \eqref{eq:commutatorForms} are independent of $x$,
\[
\begin{aligned}
  |\langle \be'\cdot\curl z,q\rangle|
  &=|\langle z,\curl(\be' q)\rangle|
    \le C|\be'|\|z\|_2\|\nabla q\|_2,\\
  |\langle (\be_i'\be_j+\be_i\be_j')\partial_i z^j,q\rangle|
  &\le C|\be'|\|z\|_2\|\nabla q\|_2.
\end{aligned}
\]
Applying this with $(z,q)=(u,\omega),(u,d),(b,j),(b,h)$, using the energy bound
\eqref{eq:energy}, and then Young's inequality together with
$\D_{\mu,\eta}\ge\lambda\D$, gives \eqref{eq:commutatorLemmaBound}.
\end{proof}

\section{The horizontal current Jacobian estimate}\label{sec:nullform}

This section contains the main MHD estimate needed for the magnetic component
and current-density part of the criterion.  The key point is that the
transverse magnetic gradient can be split
into a current part measured in $L^p$ and a magnetic-component part measured in
$L^3$.

\begin{lemma}\label{lem:magneticHodgeL3}
Let $\Div b=0$, $3\le p<\infty$, and $(\ta,\nv,\be)$ be an
orthonormal frame.  On every fixed time slice,
\begin{equation}\label{eq:magneticHodgeSplit}
  \nabla_{\be^\perp}b^{\be^\perp}
  =\mathcal R_{\be^\perp}j^\be
   +\mathcal S_{\be^\perp}d_b^\be,
  \qquad
  j^\be=\omega_b^\be,
  \quad d_b^\be=\partial_\be b^\be,
\end{equation}
where $\mathcal R_{\be^\perp}$ and $\mathcal S_{\be^\perp}$ are finite families
of horizontal Calder\'on--Zygmund operators, uniformly with respect to the
orthonormal frame.  Moreover,
\begin{equation}\label{eq:magneticHodgeLpL3}
  \|\mathcal R_{\be^\perp}j^\be\|_{L^p}
  \le C_p\|j^\be\|_{L^p},
  \qquad
  \|\mathcal S_{\be^\perp}d_b^\be\|_{L^3}
  \le C\|b^\be\|_{\dot H^{3/2}}.
\end{equation}
In particular,
\begin{equation}\label{eq:magneticHodgeL2}
  \|\nabla_{\be^\perp}b^{\be^\perp}\|_{L^2}
  \le C\bigl(\|j^\be\|_{L^2}+\|d_b^\be\|_{L^2}\bigr).
\end{equation}
\end{lemma}

\begin{proof}
Since $\Div b=0$, the horizontal vector field
$b^{\be^\perp}=(b^\ta,b^\nv)$ satisfies
\[
  \nabla_{\be^\perp}\cdot b^{\be^\perp}=-d_b^\be,
  \qquad
  \nabla_{\be^\perp}^{\perp}\cdot b^{\be^\perp}=j^\be.
\]
The two-dimensional Hodge representation in the transverse variables gives
\eqref{eq:magneticHodgeSplit}.  The operators
$\mathcal R_{\be^\perp}$ and $\mathcal S_{\be^\perp}$ are understood fiberwise:
for each fixed $x_\be\in\R_\be$ they act as ordinary two-dimensional
Calder\'on--Zygmund operators on the plane $\R^2_{\be^\perp}$.  If
$\mathcal T_{\be^\perp}$ denotes any operator in these finite families, then
Fubini's theorem and the two-dimensional Calder\'on--Zygmund theorem imply
\[
  \|\mathcal T_{\be^\perp}f\|_{L^q(\R^3)}\le C_q\|f\|_{L^q(\R^3)},
  \qquad 1<q<\infty,
\]
with $C_q$ independent of the orthonormal frame.  This proves the $L^p$ bound
for the current part and the $L^2$ bound.  For the divergence part,
$d_b^\be=\partial_\be b^\be$ and $\dot H^{1/2}(\R^3)\hookrightarrow L^3(\R^3)$
give
\[
  \|d_b^\be\|_{L^3}
  \lesssim \|\partial_\be b^\be\|_{\dot H^{1/2}}
  \lesssim \|b^\be\|_{\dot H^{3/2}}.
\]
This proves \eqref{eq:magneticHodgeLpL3}.
\end{proof}

\begin{lemma}\label{lem:nullform}
Let $3\le p<\infty$ and $\gamma_p=2p/(2p-3)$.  For every $\eps>0$,
\begin{equation}\label{eq:nullEstimate}
\left|
\int_{\R^3}2\sum_{\ell\in\{\ta,\nv\}}J_\be(b^\ell,u^\ell)\,j\,dx
\right|
\le
\eps\D_{\mu,\eta}
+C_{\eps,\lambda,p}\bigl(\|j(t)\|_{L^p}^{\gamma_p}
+\|b^\be(t)\|_{\dot H^{3/2}}^2\bigr)\F(t).
\end{equation}
\end{lemma}

\begin{proof}
By the definition of $J_\be$,
\begin{equation}\label{eq:nullStart}
\left|
\int2\sum_{\ell\in\{\ta,\nv\}}J_\be(b^\ell,u^\ell)j\,dx
\right|
\lesssim
\int |\nabla_{\be^\perp}b^{\be^\perp}|\,|\nabla_{\be^\perp}u^{\be^\perp}|\,|j|\,dx .
\end{equation}
Use the split Hodge formula \eqref{eq:magneticHodgeSplit}.  The divergence part
is estimated exactly at the endpoint $L^3$ level:
\begin{equation}\label{eq:nullDivergencePart}
\begin{aligned}
&\int |\mathcal S_{\be^\perp}h|\,
     |\nabla_{\be^\perp}u^{\be^\perp}|\,|j|\,dx       \\
&\quad\lesssim
\|b^\be\|_{\dot H^{3/2}}
\|\nabla_{\be^\perp}u^{\be^\perp}\|_2
\|j\|_{L^6}                                                   \\
&\quad\lesssim
\|b^\be\|_{\dot H^{3/2}}\F^{1/2}\D_{\mu,\eta}^{1/2}.
\end{aligned}
\end{equation}
For the current part set
\[
  r_p=\frac{2p}{p-2},
  \qquad
  \frac1p+\frac12+\frac1{r_p}=1 .
\]
The Gagliardo--Nirenberg inequality gives
\[
  \|j\|_{L^{r_p}}
  \lesssim
  \|j\|_2^{1-3/p}\|\nabla j\|_2^{3/p}.
\]
Hence
\begin{equation}\label{eq:nullCurrentPart}
\begin{aligned}
&\int |\mathcal R_{\be^\perp}j|\,
     |\nabla_{\be^\perp}u^{\be^\perp}|\,|j|\,dx       \\
&\quad\le
C_p\|j\|_{L^p}
\|\nabla_{\be^\perp}u^{\be^\perp}\|_2
\|j\|_{L^{r_p}}                                      \\
&\quad\lesssim_p
\|j\|_{L^p}\F^{1-\frac{3}{2p}}\D_{\mu,\eta}^{\frac{3}{2p}} .
\end{aligned}
\end{equation}
Since $0<3/(2p)\le1/2$, Young's inequality yields
\begin{equation}\label{eq:nullYoungSerrin}
\|j\|_{L^p}\F^{1-\frac{3}{2p}}\D_{\mu,\eta}^{\frac{3}{2p}}
\le
\eps\D_{\mu,\eta}
+C_{\eps,\lambda,p}\|j\|_{L^p}^{\gamma_p}\F .
\end{equation}
The term in \eqref{eq:nullDivergencePart} is bounded by
$\eps\D_{\mu,\eta}
+C_{\eps,\lambda}\|b^\be\|_{\dot H^{3/2}}^2\F$.  Combining these estimates proves
\eqref{eq:nullEstimate}.
\end{proof}

\begin{remark}
The estimate above explains why the scalar magnetic component $b^\be$ alone is
not enough in the present method.  The factor
$\nabla_{\be^\perp}b^{\be^\perp}$ in \eqref{eq:nullStart} is determined by both
the horizontal divergence $-\partial_\be b^\be$ and the horizontal curl
$j^\be=\be\cdot\curl b$.  The first quantity follows from $b^\be$, whereas the
second is independent information and is precisely the current-density component
assumed in Theorem~\ref{thm:main}.
\end{remark}

\section{Moving-frame a priori estimates and horizontal-gradient bounds}\label{sec:moving_apriori}

\begin{proposition}\label{prop:movingEstimate}
Let $(u,b)$ be a smooth solution of \eqref{eq:MHD} on a continuity interval of
the frame.  For every $0<\sigma\le1/5$,
\begin{equation}\label{eq:mainMovingEstimate}
\begin{aligned}
\frac{d}{dt}\F(t)+\frac12\D_{\mu,\eta}(t)
&\le
C_{\lambda,p}E_0\bigl(|\ta'(t)|^2+|\nv'(t)|^2+|\be'(t)|^2\bigr)\\
&\quad
+C_{\lambda,p}\Bigl(\|u^\be(t)\|_{\dot H^{3/2}}^2
+\|b^\be(t)\|_{\dot H^{3/2}}^2
+\|j(t)\|_{L^p}^{\gamma_p}\Bigr)\F(t)                 \\
&\quad
+\frac{C_{\lambda,p}}{\sigma}
\F(t)^{\frac1{1-\sigma}}
\A(t)^{\frac{\sigma}{1-\sigma}}                                      \\
&\qquad\qquad\times
\bigl(\|u^\be(t)\|_{\dot H^{3/2}}^2
+\|b^\be(t)\|_{\dot H^{3/2}}^2\bigr)^{\frac{1-2\sigma}{1-\sigma}} .
\end{aligned}
\end{equation}
Here $C_{\lambda,p}$ may depend on $p$ and on negative powers of
$\lambda=\min\{\mu,\eta\}$.
\end{proposition}

The proof is divided into two parts.  First, we write out the MHD terms that are
not present in the Navier--Stokes argument.  Second, we insert these estimates
into the moving-frame energy identity.

\subsection{The non-null MHD terms}

In this subsection $\mathcal R$ denotes a generic degree-zero Mikhlin multiplier
with constants depending only on the dimension and on the finite list of
operators under consideration.  Its precise form may change from line to line;
the examples needed below are $\partial_i\partial_j\Delta^{-1}$,
$\partial_{\be}^2\Delta^{-1}$, $2\partial_{\be}^2\Delta^{-1}-I$, and their
adjoints.  These operators are bounded on $L^p$, $1<p<\infty$, and on the
anisotropic Sobolev and Besov spaces used in Lemma~\ref{lem:LZproduct}.

\begin{lemma}\label{lem:nonNullMHD}
Let $\mathcal E_{\rm nn}$ denote the contribution to the $L^2$ energy identity
for $(\omega,d,j,h)$ coming from all MHD nonlinearities except the horizontal
current Jacobian residual $\mathcal N_\be$.  Then, for every
$0<\sigma\le1/5$ and every $\varepsilon>0$,
\begin{equation}\label{eq:nonNullDetailedBound}
\begin{aligned}
|\mathcal E_{\rm nn}|
&\le
\varepsilon\D_{\mu,\eta}
+C_{\varepsilon,\lambda,p}\Bigl(\|u^\be(t)\|_{\dot H^{3/2}}^2
+\|b^\be(t)\|_{\dot H^{3/2}}^2+\|j(t)\|_{L^p}^{\gamma_p}\Bigr)\F(t)                    \\
&\quad
+\frac{C_{\varepsilon,\lambda,p}}{\sigma}
\F(t)^{\frac1{1-\sigma}}
\bigl(\|u^\be(t)\|_{\dot H^{3/2}}^2
+\|b^\be(t)\|_{\dot H^{3/2}}^2\bigr)^{\frac{1-2\sigma}{1-\sigma}}
\A(t)^{\frac{\sigma}{1-\sigma}} .
\end{aligned}
\end{equation}
\end{lemma}

\begin{proof}
We list the terms generated by the magnetic part of the four equations and
estimate them one by one.  For orientation, the full MHD contribution to the
energy identity is decomposed as
\[
  \mathcal E_{\rm MHD}=\mathcal E_\omega^M+\mathcal E_d^M+
  \mathcal E_j^M+\mathcal E_h^M,
\]
where, with the sign convention in Appendix~\ref{app:expansions} and with
absolute values taken in the estimates, the terms have the forms
\[
\begin{aligned}
\mathcal E_j^M
&=\int(dj-h\omega)j
  +\int(\partial_{\be} u^{\be^\perp}\cdot\nabla_{\be^\perp}^{\perp} b^\be
      -\partial_{\be} b^{\be^\perp}\cdot\nabla_{\be^\perp}^{\perp} u^\be)j
  +\int\mathcal N_\be j,\\
\mathcal E_\omega^M
&=\int(-hj+\partial_{\be} b^{\be^\perp}\cdot\nabla_{\be^\perp}^{\perp} b^\be)\omega,\\
\mathcal E_h^M
&=\int(-\partial_{\be} u^{\be^\perp}\cdot\nabla_{\be^\perp} b^\be
      +\partial_{\be} b^{\be^\perp}\cdot\nabla_{\be^\perp}u^\be)h,\\
\mathcal E_d^M
&=\int\left(\partial_{\be} b\cdot\nabla b^\be
      +\partial_{\be}^2(-\Delta)^{-1}Q_b\right)d .
\end{aligned}
\]
The current-density equation is exactly the third equation in Lemma~\ref{lem:movingScalarSystem}:
\begin{equation}\label{eq:EjSplit}
\begin{aligned}
\mathcal E_j
&=\int_{\R^3}(d j-h\omega)j\,dx                                      \\
&\quad+\int_{\R^3}
\bigl(\partial_{\be} u^{\be^\perp}\cdot\nabla_{\be^\perp}^{\perp} b^\be
      -\partial_{\be} b^{\be^\perp}\cdot\nabla_{\be^\perp}^{\perp} u^\be\bigr)j\,dx                 \\
&\quad+\int_{\R^3}\mathcal N_\be j\,dx .
\end{aligned}
\end{equation}
The last integral is excluded from $\mathcal E_{\rm nn}$ and is handled in Lemma~\ref{lem:nullform}.  The first line of \eqref{eq:EjSplit} is algebraic.  Since
\[
  \|d\|_{L^3}\lesssim \|u^\be\|_{\dot H^{3/2}},
  \qquad
  \|h\|_{L^3}=\|\partial_\be b^\be\|_{L^3}\lesssim \|b^\be\|_{\dot H^{3/2}},
\]
we have
\begin{equation}\label{eq:algCurrent}
\begin{aligned}
\left|\int d|j|^2\,dx\right|
&\le \|d\|_{L^3}\|j\|_2\|j\|_{L^6}
 \le \varepsilon\D_{\mu,\eta}
 +C_{\varepsilon,\lambda}\|u^\be\|_{\dot H^{3/2}}^2\F,\\
\left|\int h\omega j\,dx\right|
&\le \|h\|_{L^3}\|\omega\|_2\|j\|_{L^6}
 \le \varepsilon\D_{\mu,\eta}
 +C_{\varepsilon,\lambda}\|b^\be\|_{\dot H^{3/2}}^2\F.
\end{aligned}
\end{equation}
Here and below we use $\|q\|_{L^6}\lesssim\|\nabla q\|_2$ for
$q\in\{\omega,d,j,h\}$ and \eqref{eq:Dlambda} to absorb derivatives into
$\D_{\mu,\eta}$.

The second line of \eqref{eq:EjSplit} is of the mixed anisotropic type covered
by Lemma~\ref{lem:LZproduct}.  Taking $(V,Z,q)=(u,b,j)$ and
$(V,Z,q)=(b,u,j)$ gives
\begin{equation}\label{eq:mixedCurrent}
\begin{aligned}
&\left|\int
\bigl(\partial_{\be} u^{\be^\perp}\cdot\nabla_{\be^\perp}^{\perp} b^\be
      -\partial_{\be} b^{\be^\perp}\cdot\nabla_{\be^\perp}^{\perp} u^\be\bigr)j\,dx\right|       \\
&\qquad\le
\varepsilon\D_{\mu,\eta}
+\frac{C_{\varepsilon,\lambda}}{\sigma}
\F^{\frac1{1-\sigma}}
\bigl(\|u^\be\|_{\dot H^{3/2}}^2
+\|b^\be\|_{\dot H^{3/2}}^2\bigr)^{\frac{1-2\sigma}{1-\sigma}}
\A^{\frac{\sigma}{1-\sigma}} .
\end{aligned}
\end{equation}

The magnetic part of the $\omega$ equation is the magnetic analogue of the Navier--Stokes stretching term:
\begin{equation}\label{eq:EomegaMag}
\mathcal E_\omega^{M}
=
\int_{\R^3}\bigl(-h j+\partial_{\be} b^{\be^\perp}\cdot\nabla_{\be^\perp}^{\perp} b^\be\bigr)\omega\,dx .
\end{equation}
The algebraic term is estimated as in \eqref{eq:algCurrent}, while the mixed
term is estimated by Lemma~\ref{lem:LZproduct} with
$(V,Z,q)=(b,b,\omega)$.  Thus
\begin{equation}\label{eq:EomegaMagBound}
\begin{aligned}
|\mathcal E_\omega^{M}|
&\le
\varepsilon\D_{\mu,\eta}
+C_{\varepsilon,\lambda}\|b^\be\|_{\dot H^{3/2}}^2\F             \\
&\quad+
\frac{C_{\varepsilon,\lambda}}{\sigma}
\F^{\frac1{1-\sigma}}
\|b^\be\|_{\dot H^{3/2}}^{\frac{2(1-2\sigma)}{1-\sigma}}
\A^{\frac{\sigma}{1-\sigma}} .
\end{aligned}
\end{equation}

The $h=\partial_{\be} b^\be$ equation follows by applying $\partial_{\be}$ to
the $\be$-component of the induction equation.  Apart from moving-frame
commutators, its nonlinear part is
\begin{equation}\label{eq:EhMag}
\mathcal E_h
=
\int_{\R^3}
\bigl(-\partial_{\be} u^{\be^\perp}\cdot\nabla_{\be^\perp} b^\be
      +\partial_{\be} b^{\be^\perp}\cdot\nabla_{\be^\perp} u^\be\bigr)h\,dx .
\end{equation}
The apparent scalar terms $-dh$ and $hd$ cancel exactly.  Applying
Lemma~\ref{lem:LZproduct} with $(V,Z,q)=(u,b,h)$ and $(V,Z,q)=(b,u,h)$ yields
\begin{equation}\label{eq:EhMagBound}
\begin{aligned}
|\mathcal E_h|
&\le
\varepsilon\D_{\mu,\eta}                                      \\
&\quad+
\frac{C_{\varepsilon,\lambda}}{\sigma}
\F^{\frac1{1-\sigma}}
\bigl(\|u^\be\|_{\dot H^{3/2}}^2
+\|b^\be\|_{\dot H^{3/2}}^2\bigr)^{\frac{1-2\sigma}{1-\sigma}}
\A^{\frac{\sigma}{1-\sigma}} .
\end{aligned}
\end{equation}

We next estimate the magnetic terms in the $d=\partial_{\be} u^\be$ equation.
With the sign convention of Appendix~\ref{app:pressureSplit}, the part
containing $b$ is written as
\begin{equation}\label{eq:EdMagStructure}
\mathfrak R_d^M
=
\partial_{\be} b\cdot\nabla b^\be
+\mathcal R_3
\sum_{\ell,m\in\{\ta,\nv,\be\}} \partial_\ell b^m\partial_m b^\ell,
\qquad \mathcal R_3:=\partial_{\be}^2(-\Delta)^{-1}.
\end{equation}
The pressure sum is decomposed as
\begin{equation}\label{eq:pressureSplitB}
\sum_{\ell,m\in\{\ta,\nv,\be\}} \partial_\ell b^m\partial_m b^\ell
=\sum_{\ell,m\in\{\ta,\nv\}}\partial_\ell b^m\partial_m b^\ell
+2\sum_{\ell\in\{\ta,\nv\}}\partial_\ell b^\be\partial_{\be} b^\ell+h^2.
\end{equation}
Pairing \eqref{eq:EdMagStructure} with $d$ gives three types of terms.

First, the direct term $\partial_{\be} b\cdot\nabla b^\be$ equals
\[
  h^2+\partial_{\be} b^{\be^\perp}\cdot\nabla_{\be^\perp} b^\be .
\]
The contribution of $h^2d$ is algebraic:
\begin{equation}\label{eq:hhd}
  \left|\int h^2d\,dx\right|
  \le \|h\|_{L^3}\|h\|_2\|d\|_{L^6}
  \le \varepsilon\D_{\mu,\eta}
  +C_{\varepsilon,\lambda}\|b^\be\|_{\dot H^{3/2}}^2\F.
\end{equation}
The contribution of
$(\partial_{\be} b^{\be^\perp}\cdot\nabla_{\be^\perp} b^\be)d$ is again
controlled by Lemma~\ref{lem:LZproduct}, with $(V,Z,q)=(b,b,d)$.

Second, consider the horizontal-horizontal part of the pressure term:
\[
  \int_{\R^3}\mathcal R
  \left(\sum_{\ell,m\in\{\ta,\nv\}}\partial_\ell b^m\partial_m b^\ell\right)d\,dx .
\]
By duality and the $L^q$ boundedness of $\mathcal R$, we split
$\nabla_{\be^\perp}b^{\be^\perp}$ according to Lemma~\ref{lem:magneticHodgeL3}.
The divergence part satisfies
\begin{equation}\label{eq:horizontalPressureBDiv}
\begin{aligned}
&\left|\int \mathcal R
  \left((\mathcal S_{\be^\perp}h)\,\nabla_{\be^\perp}b^{\be^\perp}\right)d\,dx\right|        \\
&\quad \lesssim
\|b^\be\|_{\dot H^{3/2}}\|\nabla_{\be^\perp}b^{\be^\perp}\|_2\|\mathcal R^*d\|_{L^6}
\le \varepsilon\D_{\mu,\eta}
+C_{\varepsilon,\lambda}\|b^\be\|_{\dot H^{3/2}}^2\F .
\end{aligned}
\end{equation}
For the current part, with $r_p=2p/(p-2)$,
\begin{equation}\label{eq:horizontalPressureBCurrent}
\begin{aligned}
&\left|\int \mathcal R
  \left((\mathcal R_{\be^\perp}j)\,\nabla_{\be^\perp}b^{\be^\perp}\right)d\,dx\right|        \\
&\quad \le
C_p\|j\|_{L^p}\|\nabla_{\be^\perp}b^{\be^\perp}\|_2\|\mathcal R^*d\|_{L^{r_p}} \\
&\quad \lesssim_p
\|j\|_{L^p}\F^{1-\frac{3}{2p}}\D_{\mu,\eta}^{\frac{3}{2p}}
\le \varepsilon\D_{\mu,\eta}+C_{\varepsilon,\lambda,p}\|j\|_{L^p}^{\gamma_p}\F .
\end{aligned}
\end{equation}
Thus
\begin{equation}\label{eq:horizontalPressureB}
\left|\int \mathcal R
  \left(\sum_{\ell,m\in\{\ta,\nv\}}\partial_\ell b^m\partial_m b^\ell\right)d\,dx\right|
\le \varepsilon\D_{\mu,\eta}
+C_{\varepsilon,\lambda,p}\bigl(\|j\|_{L^p}^{\gamma_p}
+\|b^\be\|_{\dot H^{3/2}}^2\bigr)\F .
\end{equation}
We used Lemma~\ref{lem:magneticHodgeL3}, the Hodge estimate
$\|\nabla_{\be^\perp}b^{\be^\perp}\|_2\lesssim\|j\|_2+\|h\|_2\lesssim\F^{1/2}$,
and the interpolation
$\|d\|_{L^{r_p}}\lesssim \|d\|_2^{1-3/p}\|\nabla d\|_2^{3/p}$.

Third, by \eqref{eq:pressureSplitB}, the pressure terms with at least one
vertical index are finite sums of
\[
  \int \mathcal R(h^2)d\,dx,
  \qquad
  \int \mathcal R(\partial_\ell b^\be\,\partial_{\be} b^\ell)d\,dx,
  \qquad \ell\in\{\ta,\nv\} .
\]
The first one is bounded exactly as \eqref{eq:hhd}, after moving $\mathcal R$
onto $d$.  The second one is bounded by Lemma~\ref{lem:LZproduct} with $q=d$
and the zero-order multiplier chosen as $\mathcal R^*$.  Therefore
\begin{equation}\label{eq:EdMagBound}
\begin{aligned}
\left|\int_{\R^3}\mathfrak R_d^M d\,dx\right|
&\le
\varepsilon\D_{\mu,\eta}
+C_{\varepsilon,\lambda,p}\bigl(\|j\|_{L^p}^{\gamma_p}
+\|b^\be\|_{\dot H^{3/2}}^2\bigr)\F                 \\
&\quad+
\frac{C_{\varepsilon,\lambda}}{\sigma}
\F^{\frac1{1-\sigma}}
\|b^\be\|_{\dot H^{3/2}}^{\frac{2(1-2\sigma)}{1-\sigma}}
\A^{\frac{\sigma}{1-\sigma}} .
\end{aligned}
\end{equation}

Combining \eqref{eq:algCurrent}--\eqref{eq:EdMagBound}, choosing
$\varepsilon>0$ sufficiently small, and using
\[
  \|u^\be\|_{\dot H^{3/2}}^{\frac{2(1-2\sigma)}{1-\sigma}}+
  \|b^\be\|_{\dot H^{3/2}}^{\frac{2(1-2\sigma)}{1-\sigma}}
  \le C\bigl(\|u^\be\|_{\dot H^{3/2}}^2
  +\|b^\be\|_{\dot H^{3/2}}^2\bigr)^{\frac{1-2\sigma}{1-\sigma}},
\]
we obtain \eqref{eq:nonNullDetailedBound}.
\end{proof}

We next record the corresponding velocity-only estimate.  This is the
moving-frame version of the Navier--Stokes contribution.

\begin{lemma}\label{lem:NSContribution}
Let $\mathfrak R^{NS}_\omega$ and $\mathfrak R^{NS}_d$ denote the $u$-only terms
in the first two equations of \eqref{eq:structure}.  Then, for every
$0<\sigma\le1/5$ and every $\varepsilon>0$,
\begin{equation}\label{eq:NSContributionLemma}
\begin{aligned}
|\langle\mathfrak R^{NS}_\omega,\omega\rangle|+
|\langle\mathfrak R^{NS}_d,d\rangle|
&\le
\varepsilon\D_{\mu,\eta}
+C_{\varepsilon,\lambda}\|u^\be(t)\|_{\dot H^{3/2}}^2\F(t)                         \\
&\quad +\frac{C_{\varepsilon,\lambda}}{\sigma}
\F(t)^{\frac1{1-\sigma}}
\|u^\be(t)\|_{\dot H^{3/2}}^{\frac{2(1-2\sigma)}{1-\sigma}}
\A(t)^{\frac{\sigma}{1-\sigma}} .
\end{aligned}
\end{equation}
\end{lemma}

\begin{proof}
After freezing the frame, the $u$-only part of the $(\omega,d)$ system is the
same scalar system used in the moving one-velocity-component Navier--Stokes
argument of Liu--Zhang.  More explicitly,
Appendix~\ref{app:fixedScalarEquations} gives
\[
  \mathfrak R_\omega^{NS}
  =d\omega-\partial_{\be} u^{\be^\perp}\cdot\nabla_{\be^\perp}^{\perp} u^\be,
\]
and, with $Q_u=\sum_{\ell,m}\partial_\ell u^m\partial_m u^\ell$,
\[
  \mathfrak R_d^{NS}
  =-\partial_{\be} u\cdot\nabla u^\be
   -\partial_{\be}^2(-\Delta)^{-1}Q_u,
\]
up to the harmless sign convention for $(-\Delta)^{-1}$.
We decompose the pressure quadratic form as
\begin{equation}\label{eq:pressureSplitU}
  Q_u=\sum_{\ell,m\in\{\ta,\nv\}}\partial_\ell u^m\partial_m u^\ell
  +2\sum_{\ell\in\{\ta,\nv\}}\partial_\ell u^\be\partial_{\be} u^\ell+d^2 .
\end{equation}
Since $d=\partial_{\be} u^\be$ and
$\nabla_{\be^\perp}\cdot u^{\be^\perp}=-d$, the horizontal
Hodge formula gives
\begin{equation}\label{eq:uHorizontalHodgeNS}
  \nabla_{\be^\perp} u^{\be^\perp}=\mathcal R_h(\omega,d),
\end{equation}
where $\mathcal R_h$ denotes a finite family of horizontal Calder\'on--Zygmund
operators.  Therefore the horizontal-horizontal pressure block, after moving
the zero-order multiplier $\partial_{\be}^2(-\Delta)^{-1}$ onto $d$, is a finite
sum of terms
\begin{equation}\label{eq:NSHHmodel}
  \int_{\R^3} q_1q_2\,\Riesz d\,dx,
  \qquad q_1,q_2\in\{\omega,d\}
\end{equation}
with $\Riesz$ a degree-zero multiplier.  Sobolev embedding and the
one-velocity-component control give
\[
  \norm{\Riesz d}{L^3}
  \lesssim \norm{d}{\dot H^{1/2}}
  \lesssim \norm{u^\be}{\dot H^{3/2}}.
\]
Thus, using also $\norm{q_2}{L^6}\lesssim\norm{\nabla q_2}{2}$,
\begin{equation}\label{eq:NSHHbound}
  \left|\int q_1q_2\,\Riesz d\,dx\right|
  \lesssim \norm{u^\be}{\dot H^{3/2}}\F^{1/2}\D^{1/2}
  \le \varepsilon\D_{\mu,\eta}
  +C_{\varepsilon,\lambda}\norm{u^\be}{\dot H^{3/2}}^2\F.
\end{equation}
The algebraic pieces coming from $d\omega$, from
$-\partial_{\be} u\cdot\nabla u^\be$, and from the $d^2$ pressure block are
estimated in the same way; for
$q=\omega$ or $d$,
\begin{equation}\label{eq:NSAlgebraicBound}
  \norm{d}{L^3}\norm{q}{2}\norm{q}{L^6}
  \lesssim \norm{u^\be}{\dot H^{3/2}}\F^{1/2}\D^{1/2}
  \le \varepsilon\D_{\mu,\eta}
  +C_{\varepsilon,\lambda}\norm{u^\be}{\dot H^{3/2}}^2\F.
\end{equation}
It remains to estimate the horizontal-vertical pressure block in
\eqref{eq:pressureSplitU} and the mixed direct term
$\partial_{\be} u^{\be^\perp}\cdot\nabla_{\be^\perp}^\#u^\be$.  These terms are
finite sums of
\begin{equation}\label{eq:NSMixedModel}
  \int_{\R^3}
  (\partial_{\be} u^{\be^\perp}\cdot\nabla_{\be^\perp}^\#u^\be)\Riesz q\,dx,
  \qquad q\in\{\omega,d\},
\end{equation}
where $\nabla_{\be^\perp}^\#$ denotes either $\nabla_{\be^\perp}$ or $\nabla_{\be^\perp}^{\perp}$.
Lemma~\ref{lem:LZproduct}, applied with $V=Z=u$, yields
\begin{equation}\label{eq:NSMixedBound}
  \left|\int_{\R^3}(\partial_{\be} u^{\be^\perp}\cdot\nabla_{\be^\perp}^\#u^\be)\Riesz q\,dx\right|
  \le \varepsilon\D_{\mu,\eta}
  +\frac{C_{\varepsilon,\lambda}}{\sigma}
  \F^{\frac1{1-\sigma}}
  \norm{u^\be}{\dot H^{3/2}}^{\frac{2(1-2\sigma)}{1-\sigma}}
  \A^{\frac{\sigma}{1-\sigma}} .
\end{equation}
Combining \eqref{eq:NSHHbound}, \eqref{eq:NSAlgebraicBound}, and
\eqref{eq:NSMixedBound}, summing over the finite list of terms, and choosing
$\varepsilon$ small enough proves \eqref{eq:NSContributionLemma}.
\end{proof}

\subsection{Proof of Proposition~\ref{prop:movingEstimate}}

\begin{proof}[Proof of Proposition~\ref{prop:movingEstimate}]
We multiply the four equations in \eqref{eq:structure} by $\omega,d,j,h$,
respectively, integrate over $\R^3$, and sum the results.

The transport terms vanish or cancel.  Since $\Div u=0$,
\[
  \int (u\cdot\nabla\omega)\omega\,dx=0,
  \qquad
  \int (u\cdot\nabla j)j\,dx=0,
\]
and the same holds for $d$ and $h$.  The cross-advection terms cancel in pairs because $\Div b=0$:
\[
  -\int (b\cdot\nabla j)\omega\,dx
  -\int (b\cdot\nabla\omega)j\,dx=0,
  \qquad
  -\int (b\cdot\nabla h)d\,dx
  -\int (b\cdot\nabla d)h\,dx=0.
\]
The diffusion gives $\D_{\mu,\eta}$.

The moving-frame commutators are controlled by Lemma~\ref{lem:commutators}:
\begin{equation}\label{eq:commEstimate}
  |\langle \mathcal C_\omega,\omega\rangle|+
  |\langle \mathcal C_d,d\rangle|+
  |\langle \mathcal C_j,j\rangle|+
  |\langle \mathcal C_h,h\rangle|
  \le \varepsilon\D_{\mu,\eta}
  +C_{\varepsilon,\lambda} E_0\bigl(|\ta'(t)|^2+|\nv'(t)|^2+|\be'(t)|^2\bigr).
\end{equation}

The Navier--Stokes part is controlled by Lemma~\ref{lem:NSContribution}:
\begin{equation}\label{eq:NSContribution}
\begin{aligned}
|\langle\mathfrak R^{NS}_\omega,\omega\rangle|+
|\langle\mathfrak R^{NS}_d,d\rangle|
&\le
\varepsilon\D_{\mu,\eta}
+C_{\varepsilon,\lambda}\|u^\be(t)\|_{\dot H^{3/2}}^2\F(t)                         \\
&\quad +\frac{C_{\varepsilon,\lambda}}{\sigma}
\F(t)^{\frac1{1-\sigma}}
\|u^\be(t)\|_{\dot H^{3/2}}^{\frac{2(1-2\sigma)}{1-\sigma}}
\A(t)^{\frac{\sigma}{1-\sigma}} .
\end{aligned}
\end{equation}

The non-null MHD terms are controlled by Lemma~\ref{lem:nonNullMHD}.  The
remaining MHD term is the horizontal current Jacobian residual.
Lemma~\ref{lem:nullform}, with $\varepsilon$ chosen small enough, gives
\begin{equation}\label{eq:nullUse}
  |\langle \mathcal N_\be,j\rangle|
  \le \varepsilon\D_{\mu,\eta}
  +C_{\varepsilon,\lambda,p}\bigl(\|j(t)\|_{L^p}^{\gamma_p}
  +\|b^\be(t)\|_{\dot H^{3/2}}^2\bigr)\F(t).
\end{equation}
Combining \eqref{eq:commEstimate}, \eqref{eq:NSContribution},
\eqref{eq:nonNullDetailedBound}, and \eqref{eq:nullUse}, and then choosing
$\varepsilon>0$ sufficiently small in the four absorbed estimates, yields
\eqref{eq:mainMovingEstimate}.
\end{proof}

\subsection{Boundedness of the moving horizontal gradients}\label{sec:horizontal_gradients}

Throughout this subsection we work on one continuity interval of the moving
frame.  If $\be$ has jump times, the same argument is applied separately on each
continuity interval, after translating the initial time to the left endpoint.
At a jump time the strong solution is continuous in $H^1$, so the reinitialized
quantities $\F$ and $\A$ are finite in the new orthonormal frame; since the
number of jumps is finite, a finite induction gives the asserted bounds on the
whole interval.

Assume the criterion quantity is finite on $[0,T]$:
\begin{equation}\label{eq:finiteQT}
  \int_0^T\Bigl(
  \|u^\be(t)\|_{\dot H^{3/2}}^2
  +\|b^\be(t)\|_{\dot H^{3/2}}^2
  +\|j(t)\|_{L^p}^{\gamma_p}\Bigr)\,dt<\infty,
  \qquad T<T_*.
\end{equation}
Since $\be\in\Om(T_*)$ and the number of jumps is finite,
\begin{equation}\label{eq:finiteK}
  \int_0^T\bigl(|\ta'(t)|^2+|\nv'(t)|^2+|\be'(t)|^2\bigr)\,dt<\infty.
\end{equation}
By the energy inequality,
\begin{equation}\label{eq:finiteA}
  \int_0^T\A(t)\,dt
  \le \lambda^{-1}E_0<\infty.
\end{equation}
We now apply Lemma~\ref{lem:iteration} to Proposition~\ref{prop:movingEstimate} with
\[
\begin{aligned}
  f&=\F,
  \qquad
  Y=\A,\\
  X(t)&=\|u^\be(t)\|_{\dot H^{3/2}}^2
       +\|b^\be(t)\|_{\dot H^{3/2}}^2,\\
  V_1(t)&=X(t)+\|j(t)\|_{L^p}^{\gamma_p},\\
  V_2(t)&=|\ta'(t)|^2+|\nv'(t)|^2+|\be'(t)|^2 .
\end{aligned}
\]
The hypotheses of the lemma follow from \eqref{eq:finiteQT}, \eqref{eq:finiteK},
and \eqref{eq:finiteA}.  Hence
\begin{equation}\label{eq:Fbounded}
  \sup_{0\le t<T}\F(t)<\infty.
\end{equation}
Fix, for instance, $\sigma=1/5$.  By H\"older's inequality,
\[
\begin{aligned}
&\int_0^T
  \bigl(\|u^\be(t)\|_{\dot H^{3/2}}^2
  +\|b^\be(t)\|_{\dot H^{3/2}}^2\bigr)^{\frac{1-2\sigma}{1-\sigma}}
  \A(t)^{\frac{\sigma}{1-\sigma}}\,dt                                      \\
&\quad\le
  \left(\int_0^T\bigl(\|u^\be(t)\|_{\dot H^{3/2}}^2
  +\|b^\be(t)\|_{\dot H^{3/2}}^2\bigr)\,dt\right)^{\frac{1-2\sigma}{1-\sigma}} \\
&\qquad\times
  \left(\int_0^T\A(t)\,dt\right)^{\frac{\sigma}{1-\sigma}}
  <\infty .
\end{aligned}
\]
Integrating \eqref{eq:mainMovingEstimate} and using \eqref{eq:Fbounded} gives
\begin{equation}\label{eq:Dintegrable}
  \int_0^T\D_{\mu,\eta}(t)\,dt<\infty.
\end{equation}
By the Hodge estimate \eqref{eq:HodgeEstimate},
\begin{equation}\label{eq:horizontalControl}
\begin{aligned}
&\|\nabla_{\be^\perp}u^{\be^\perp}\|_{L^\infty(0,T;L^2)}^2
+\|\nabla_{\be^\perp}b^{\be^\perp}\|_{L^\infty(0,T;L^2)}^2       \\
&\quad
+\|\nabla\nabla_{\be^\perp}u^{\be^\perp}\|_{L^2(0,T;L^2)}^2
+\|\nabla\nabla_{\be^\perp}b^{\be^\perp}\|_{L^2(0,T;L^2)}^2
<\infty.
\end{aligned}
\end{equation}

\section{Proof of the main theorem}\label{sec:proof_main}

\subsection{The \texorpdfstring{$H^1$}{H1} closure}\label{sec:H1_closure}

We now turn to the proof of Theorem~\ref{thm:main}.  The first step is to close
the estimate for the full $H^1$ norm, using the horizontal-gradient bounds from
Subsection~\ref{sec:horizontal_gradients}.  We begin with a structural bound for
the velocity cubic term, which is the only part requiring a separate argument.

\begin{lemma}\label{lem:cubicVelocity}
Let $v$ be smooth and divergence free.  In the moving orthonormal frame,
\begin{equation}\label{eq:cubicVelocityLemma}
\left|
\sum_{\ell,m\in\{\ta,\nv,\be\}}\int_{\R^3}
(\partial_\ell v\cdot\nabla v^m)\,\partial_\ell v^m\,dx
\right|
\le C\int_{\R^3}|\nabla v|^2|\nabla_{\be^\perp}v^{\be^\perp}|\,dx .
\end{equation}
\end{lemma}

\begin{proof}
Let $\alpha\in\{\ta,\nv\}$ denote a horizontal direction.  Since $\Div v=0$,
\[
  \partial_{\be}v^\be=-\partial_{\ta}v^\ta-\partial_{\nv}v^\nv.
\]
We examine the possible positions of the pair
$(\ell,m)\in\{\ta,\nv,\be\}\times\{\ta,\nv,\be\}$.
If $\ell,m\in\{\ta,\nv\}$, then $\partial_\ell v^m$ is a component of
$\nabla_{\be^\perp}v^{\be^\perp}$.  If $\ell=m=\be$, then
\[
  (\partial_{\be}v\cdot\nabla v^\be)\partial_{\be}v^\be
  =-(\partial_{\be}v\cdot\nabla v^\be)(\partial_{\ta}v^\ta+
    \partial_{\nv}v^\nv),
\]
which again contains the horizontal-gradient factor.  If $\ell=\be$ and
$m=\alpha$, then
\[
\begin{aligned}
(\partial_{\be}v\cdot\nabla v^\alpha)\partial_{\be}v^\alpha
&=\bigl(\partial_{\be}v^{\be^\perp}\cdot\nabla_{\be^\perp}v^\alpha
      +\partial_{\be}v^\be\partial_{\be}v^\alpha\bigr)\partial_{\be}v^\alpha  \\
&=\bigl(\partial_{\be}v^{\be^\perp}\cdot\nabla_{\be^\perp}v^\alpha
      -(\nabla_{\be^\perp}\cdot v^{\be^\perp})\partial_{\be}v^\alpha\bigr)
      \partial_{\be}v^\alpha .
\end{aligned}
\]
Finally, if $\ell=\alpha$ and $m=\be$, then
\[
\begin{aligned}
(\partial_\alpha v\cdot\nabla v^\be)\partial_\alpha v^\be
&=\bigl(\partial_\alpha v^{\be^\perp}\cdot\nabla_{\be^\perp}v^\be
      +\partial_\alpha v^\be\partial_{\be}v^\be\bigr)\partial_\alpha v^\be  \\
&=\bigl(\partial_\alpha v^{\be^\perp}\cdot\nabla_{\be^\perp}v^\be
      -(\nabla_{\be^\perp}\cdot v^{\be^\perp})\partial_\alpha v^\be\bigr)
      \partial_\alpha v^\be .
\end{aligned}
\]
The four alternatives above exhaust all pairs
$(\ell,m)\in\{\ta,\nv,\be\}\times\{\ta,\nv,\be\}$,
because each index is either transverse or equal to $\be$.  In each case one
factor is a component of $\nabla_{\be^\perp}v^{\be^\perp}$ and the remaining two
factors are bounded by $|\nabla v|$.  The constant is independent of the time
slice because the frame is orthonormal.  Summing over the finite frame proves
\eqref{eq:cubicVelocityLemma}.

\end{proof}

We now combine this velocity structure with the cancellations in the coupled
MHD terms.

\begin{lemma}\label{lem:MHDanisotropicH1}
For smooth divergence-free vector fields $u$ and $b$,
\begin{equation}\label{eq:MHDanisotropicH1}
\begin{aligned}
&\left| -\sum_{\ell,m\in\{\ta,\nv,\be\}}\int
(\partial_\ell u\cdot\nabla u^m)\partial_\ell u^m\,dx
 -\sum_{\ell,m\in\{\ta,\nv,\be\}}\int
(\partial_\ell u\cdot\nabla b^m)\partial_\ell b^m\,dx \right. \\
&\qquad\left.
 +\sum_{\ell,m\in\{\ta,\nv,\be\}}\int
(\partial_\ell b\cdot\nabla b^m)\partial_\ell u^m\,dx
 +\sum_{\ell,m\in\{\ta,\nv,\be\}}\int
(\partial_\ell b\cdot\nabla u^m)\partial_\ell b^m\,dx
\right|                                                        \\
&\qquad\le
C\int_{\R^3}\bigl(|\nabla_{\be^\perp}u|+|\nabla_{\be^\perp}b|\bigr)
\bigl(|\nabla u|^2+|\nabla b|^2\bigr)\,dx .
\end{aligned}
\end{equation}
\end{lemma}

\begin{proof}
The estimate is a frame-invariant form of the standard anisotropic $H^1$
structure for the incompressible MHD system; compare the horizontal-gradient
estimates in \cite{CaoWuMHD,JiaZhou2012Partial}.  We give the algebra in the
present moving frame.  The velocity self-interaction is covered by
Lemma~\ref{lem:cubicVelocity}.  For the remaining three sums, freeze the
orthonormal frame and expand the dot products.  Each cubic monomial is of the
form
\[
  \partial_\ell U^k\,\partial_k V^m\,\partial_\ell W^m,
  \qquad
  U,V,W\in\{u,b\},
  \quad \ell,k,m\in\{\ta,\nv,\be\}.
\]
If $\ell$ is transverse, then both the first and third factors carry a
transverse derivative.  If $\ell=\be$ but $k$ is transverse, then the middle
factor carries a transverse derivative.  In either situation one factor is
bounded by $|\nabla_{\be^\perp}u|+|\nabla_{\be^\perp}b|$ and the other two are
bounded by $|\nabla u|+|\nabla b|$.  The only configuration not covered by this
observation is $\ell=k=\be$.  Then the first factor is
$\partial_\be u^\be$ or $\partial_\be b^\be$, and incompressibility gives
\[
  \partial_\be u^\be=-\nabla_{\be^\perp}\cdot u^{\be^\perp},
  \qquad
  \partial_\be b^\be=-\nabla_{\be^\perp}\cdot b^{\be^\perp}.
\]
Thus this exceptional configuration also contains a transverse derivative of
$u$ or $b$.  Consequently every monomial is bounded pointwise by
\[
  C\bigl(|\nabla_{\be^\perp}u|+|\nabla_{\be^\perp}b|\bigr)
  \bigl(|\nabla u|+|\nabla b|\bigr)^2,
\]
and this is controlled by the right-hand side of
\eqref{eq:MHDanisotropicH1}.  The constant is independent of the time slice
because the frame is orthonormal.
\end{proof}

Let
\begin{equation}\label{eq:Y}
  Y(t)=\|\nabla u(t)\|_2^2+\|\nabla b(t)\|_2^2.
\end{equation}
The identity below is first derived in a fixed Cartesian coordinate system.  We
then rewrite the resulting rotationally invariant tensor contractions in the
instantaneous orthonormal frame.  Thus no time derivative of the moving frame
appears in this $H^1$ energy identity.
Applying $\nabla$ to \eqref{eq:MHD}, testing the velocity equation with
$\nabla u$ and the magnetic equation with $\nabla b$, and summing over the
differentiating index gives
\begin{equation}\label{eq:H1exact}
\begin{aligned}
\frac12\frac{d}{dt}Y(t)
&+\mu\|\nabla^2u\|_2^2+
\eta\|\nabla^2b\|_2^2                                      \\
&= -\sum_{\ell,m\in\{\ta,\nv,\be\}}\int
(\partial_\ell u\cdot\nabla u^m)\partial_\ell u^m\,dx               \\
&\quad -\sum_{\ell,m\in\{\ta,\nv,\be\}}\int
(\partial_\ell u\cdot\nabla b^m)\partial_\ell b^m\,dx               \\
&\quad +\sum_{\ell,m\in\{\ta,\nv,\be\}}\int
(\partial_\ell b\cdot\nabla b^m)\partial_\ell u^m\,dx               \\
&\quad +\sum_{\ell,m\in\{\ta,\nv,\be\}}\int
(\partial_\ell b\cdot\nabla u^m)\partial_\ell b^m\,dx .
\end{aligned}
\end{equation}
The transport terms do not contribute to the energy identity.  The terms with
$u$ as the transport field give
\[
  \frac12\sum_{\ell\in\{\ta,\nv,\be\}}\int
  u\cdot\nabla\bigl(|\partial_\ell u|^2+|\partial_\ell b|^2\bigr)\,dx=0
\]
by $\Div u=0$.  The two terms with $b$ as the transport field cancel together,
since
\[
  \sum_{\ell\in\{\ta,\nv,\be\}}\int b\cdot\nabla(\partial_\ell b\cdot\partial_\ell u)\,dx=0,
\]
by $\Div b=0$.  The pressure term vanishes after integration by parts and
$\Div u=0$.

From \eqref{eq:H1exact} and Lemma~\ref{lem:MHDanisotropicH1},
\begin{equation}\label{eq:H1identity}
\begin{aligned}
\frac12\frac{d}{dt}Y(t)
&+\mu\|\nabla^2u\|_2^2+
\eta\|\nabla^2b\|_2^2                                      \\
&\le
C\int_{\R^3}\bigl(|\nabla_{\be^\perp}u|+|\nabla_{\be^\perp}b|\bigr)
\bigl(|\nabla u|^2+|\nabla b|^2\bigr)\,dx .
\end{aligned}
\end{equation}
Set
\[
  Z(t)=\|\nabla^2u(t)\|_2^2+\|\nabla^2b(t)\|_2^2.
\]
By H\"older's inequality and Sobolev embedding,
\begin{equation}\label{eq:H1Holder}
\begin{aligned}
&\int \bigl(|\nabla_{\be^\perp}u|+|\nabla_{\be^\perp}b|\bigr)
\bigl(|\nabla u|^2+|\nabla b|^2\bigr)\,dx                         \\
&\quad\le
C\bigl(\|\nabla_{\be^\perp}u\|_{L^3}+\|\nabla_{\be^\perp}b\|_{L^3}\bigr)
Y(t)^{1/2}Z(t)^{1/2} .
\end{aligned}
\end{equation}
We split the transverse gradients into transverse components and scalar
components:
\[
  \nabla_{\be^\perp}u=(\nabla_{\be^\perp}u^{\be^\perp},\nabla_{\be^\perp}u^\be),
  \qquad
  \nabla_{\be^\perp}b=(\nabla_{\be^\perp}b^{\be^\perp},\nabla_{\be^\perp}b^\be).
\]
The scalar pieces are controlled by the criterion quantities,
\begin{equation}\label{eq:scalarHorizontalL3}
  \|\nabla_{\be^\perp}u^\be\|_{L^3}^2
  \lesssim \|u^\be(t)\|_{\dot H^{3/2}}^2,
  \qquad
  \|\nabla_{\be^\perp}b^\be\|_{L^3}^2
  \lesssim \|b^\be(t)\|_{\dot H^{3/2}}^2.
\end{equation}
For the transverse vector pieces, interpolation gives
\begin{equation}\label{eq:horizontalVectorL3}
\begin{aligned}
\|\nabla_{\be^\perp}u^{\be^\perp}\|_{L^3}^2
&\lesssim
\|\nabla_{\be^\perp}u^{\be^\perp}\|_2
\|\nabla\nabla_{\be^\perp}u^{\be^\perp}\|_2,\\
\|\nabla_{\be^\perp}b^{\be^\perp}\|_{L^3}^2
&\lesssim
\|\nabla_{\be^\perp}b^{\be^\perp}\|_2
\|\nabla\nabla_{\be^\perp}b^{\be^\perp}\|_2.
\end{aligned}
\end{equation}
Combining \eqref{eq:H1Holder}--\eqref{eq:horizontalVectorL3} and applying
Young's inequality, we obtain
\begin{equation}\label{eq:H1Gronwall}
\frac{d}{dt}Y(t)
+\mu\|\nabla^2u\|_2^2
+\eta\|\nabla^2b\|_2^2
\le
C_\lambda V(t)Y(t),
\end{equation}
where
\begin{equation}\label{eq:V}
\begin{aligned}
V(t)&=\|u^\be(t)\|_{\dot H^{3/2}}^2
+\|b^\be(t)\|_{\dot H^{3/2}}^2
+\|\nabla_{\be^\perp}u^{\be^\perp}(t)\|_2
 \|\nabla\nabla_{\be^\perp}u^{\be^\perp}(t)\|_2       \\
&\quad
+\|\nabla_{\be^\perp}b^{\be^\perp}(t)\|_2
 \|\nabla\nabla_{\be^\perp}b^{\be^\perp}(t)\|_2 .
\end{aligned}
\end{equation}
By \eqref{eq:finiteQT} and \eqref{eq:horizontalControl}, $V\in L^1(0,T)$.
Gronwall's inequality gives
\begin{equation}\label{eq:H1Bound}
  \sup_{0\le t<T}Y(t)
  +\int_0^T\bigl(\mu\|\nabla^2u\|_2^2+
  \eta\|\nabla^2b\|_2^2\bigr)\,dt
  <\infty.
\end{equation}

\subsection{Completion of the proof}\label{subsec:proof_completion}

We use the following standard continuation criterion: the local $H^1$ theory and
the usual blow-up alternative for strong MHD solutions imply that, if
\[
  \sup_{0<t<T}\bigl(\|u(t)\|_{H^1}+\|b(t)\|_{H^1}\bigr)<\infty,
\]
then the solution extends beyond $T$; see
\cite{DuvautLions,SermangeTemam}.

\begin{lemma}\label{lem:smoothPassage}
Let $(u,b)$ be an $H^1$ strong solution of \eqref{eq:MHD} on $[0,T]$, and let
$\be\in\Om(T)$.  If
\[
  \int_0^T\Bigl(
  \|u^\be(t)\|_{\dot H^{3/2}}^2
  +\|b^\be(t)\|_{\dot H^{3/2}}^2
  +\|j(t)\|_{L^p}^{\gamma_p}\Bigr)\,dt<\infty,
\]
then the integral estimates \eqref{eq:Fbounded}, \eqref{eq:Dintegrable}, and
\eqref{eq:H1Bound}, proved above for smooth solutions, also hold for $(u,b)$.
\end{lemma}

\begin{proof}
Let $I'\Subset I$ be a compact subinterval of a continuity interval of the
frame.  Apply the Friedrichs projection $S_N$ in space to \eqref{eq:MHD}.  The
projected system is smooth on $I'$, and the estimates of
Sections~\ref{sec:vorticity_current}--\ref{sec:moving_apriori} and
Subsection~\ref{sec:H1_closure} apply with constants independent of $N$.  Since
$S_N(u,b)\to(u,b)$ in
$C(I';H^1)\cap L^2(I';\dot H^2)$, the transport, stretching, pressure, and
diffusion terms converge in the corresponding integral identities.  The
Calder\'on--Zygmund, Hodge, and anisotropic product estimates are stable under
$S_N$, and the criterion coefficients are bounded by the original ones because
$S_N$ is bounded on $\dot H^{3/2}$ and on $L^p$.  Passing to the limit, using
weak lower semicontinuity for the dissipative terms, gives the same estimates on
$I'$.  Letting $I'$ exhaust $I$ proves the claim on each continuity interval.
The finitely many jump times are then handled by the $H^1$ continuity of the
solution.
\end{proof}

\begin{proof}[Proof of Theorem~\ref{thm:main}]
We prove the continuation statement; the blow-up statement follows by
contraposition.  Fix $T<T_*$.  By Lemma~\ref{lem:smoothPassage}, the estimates
established in
Sections~\ref{sec:vorticity_current}--\ref{sec:moving_apriori} and
Subsection~\ref{sec:H1_closure} apply to the present $H^1$ strong solution on
each continuity interval of $\be$.  The finitely many frame jumps are handled by
the $H^1$ continuity of the solution.  Hence \eqref{eq:Fbounded},
\eqref{eq:Dintegrable}, and \eqref{eq:H1Bound} hold on every interval
$[0,T]\subset[0,T_*)$.

Assume now that \eqref{eq:criterionAssumption} is finite on $[0,T_*)$.  The
constants in the preceding estimates depend only on the energy, on the total
integrals of the three criterion norms, on
\[
  \int_0^T\bigl(|\ta'(t)|^2+|\nv'(t)|^2+|\be'(t)|^2\bigr)\,dt
  \qquad\hbox{and}\qquad
  \int_0^T\A(t)\,dt,
\]
on the fixed exponent $p$, and on the finite list of frame jumps.  Hence they
are independent of $T<T_*$.  Letting $T\uparrow T_*$ gives
\begin{equation}\label{eq:uniformH1}
  \sup_{0<t<T_*}\bigl(\|u(t)\|_{H^1}+\|b(t)\|_{H^1}\bigr)<\infty.
\end{equation}
The standard continuation criterion then extends the solution beyond $T_*$.  This
contradicts maximality.  Therefore finite-time breakdown implies
\eqref{eq:blowupCriterion}.
\end{proof}

\appendix

\section{Algebraic expansions in the moving frame}\label{app:expansions}

This appendix records the algebraic computations used in the proof.  On a
continuity interval we freeze the orthonormal frame and write all components and
derivatives directly in the notation $(\ta,\nv,\be)$.  The time-dependence of
the frame only produces the commutators denoted by
$\mathcal C_\omega,\mathcal C_d,\mathcal C_j,\mathcal C_h$ in the main text.

\subsection{Scalar equations in the frozen moving frame}\label{app:fixedScalarEquations}
On a continuity interval, after freezing the moving frame, set
\[
  \omega=\partial_{\ta} u^\nv-\partial_{\nv} u^\ta,\qquad
  j=\partial_{\ta} b^\nv-\partial_{\nv} b^\ta,\qquad
  d=\partial_{\be} u^\be,\qquad h=\partial_{\be} b^\be .
\]
The $\be$-component of the vorticity equation gives
\begin{equation}\label{eq:appendixOmegaEquation}
\begin{aligned}
\partial_t\omega+u\cdot\nabla\omega-b\cdot\nabla j-\mu\Delta\omega
&=d\omega-\partial_{\be} u^{\be^\perp}\cdot\nabla_{\be^\perp}^{\perp} u^\be
  -hj+\partial_{\be} b^{\be^\perp}\cdot\nabla_{\be^\perp}^{\perp} b^\be .
\end{aligned}
\end{equation}
Indeed, the horizontal derivatives of $u^\be$ cancel in the self-interaction,
and the same computation with $b$ gives the last two terms.

Applying $\partial_{\be}$ to the $\be$-component of the velocity equation gives
\begin{equation}\label{eq:appendixDEquation}
\partial_td+u\cdot\nabla d-b\cdot\nabla h-\mu\Delta d
= -\partial_{\be} u\cdot\nabla u^\be
+\partial_{\be} b\cdot\nabla b^\be-\partial_{\be}^2P .
\end{equation}
The pressure is determined by
\[
  -\Delta P=Q_u-Q_b,
  \qquad
  Q_u=\sum_{\ell,m\in\{\ta,\nv,\be\}}\partial_\ell u^m\partial_m u^\ell,
  \qquad
  Q_b=\sum_{\ell,m\in\{\ta,\nv,\be\}}\partial_\ell b^m\partial_m b^\ell .
\]
The terms containing only $u$ in \eqref{eq:appendixDEquation} are the
Navier--Stokes terms; the terms containing $b$ are split in
Appendix~\ref{app:pressureSplit}.

Finally, applying $\partial_{\be}$ to the $\be$-component of the induction
equation yields
\begin{equation}\label{eq:appendixHEquation}
\begin{aligned}
\partial_th+u\cdot\nabla h-b\cdot\nabla d-\eta\Delta h
&=-\partial_{\be} u\cdot\nabla b^\be+\partial_{\be} b\cdot\nabla u^\be  \\
&=-\partial_{\be} u^{\be^\perp}\cdot\nabla_{\be^\perp} b^\be
  +\partial_{\be} b^{\be^\perp}\cdot\nabla_{\be^\perp} u^\be,
\end{aligned}
\end{equation}
where the scalar terms $-dh$ and $hd$ cancel exactly.  Equations
\eqref{eq:appendixOmegaEquation}--\eqref{eq:appendixHEquation}, together
with the current-density equation below, justify the structural system
\eqref{eq:structure}.

\subsection{The current-density equation}\label{app:current}
Let
\[
  \omega=\partial_{\ta} u^\nv-\partial_{\nv} u^\ta,
  \qquad j=\partial_{\ta} b^\nv-\partial_{\nv} b^\ta,
  \qquad d=\partial_{\be} u^\be,
  \qquad h=\partial_{\be} b^\be.
\]
The vector vorticity-current formulation gives
\[
\partial_tJ+u\cdot\nabla J-b\cdot\nabla\Omega-\eta\Delta J
=J\cdot\nabla u-\Omega\cdot\nabla b
+2\sum_{\ell\in\{\ta,\nv,\be\}}\nabla b^\ell\times\nabla u^\ell .
\]
Taking the $\be$-component yields
\begin{equation}\label{eq:appendixCurrentStart}
\begin{aligned}
&\partial_tj+u\cdot\nabla j-b\cdot\nabla\omega-\eta\Delta j  \\
&\quad =J\cdot\nabla u^\be-\Omega\cdot\nabla b^\be
+2\sum_{m\in\{\ta,\nv,\be\}}\bigl(\partial_{\ta} b^m\partial_{\nv} u^m-
\partial_{\nv} b^m\partial_{\ta} u^m\bigr).
\end{aligned}
\end{equation}
Using
\[
J^\ta=\partial_{\nv} b^\be-\partial_{\be} b^\nv,
\quad J^\nv=\partial_{\be} b^\ta-\partial_{\ta} b^\be,
\quad J^\be=j,
\]
and the analogous identities for $\Omega$, the stretching part equals
\begin{equation}\label{eq:stretchingCurrent}
\begin{aligned}
J\cdot\nabla u^\be-\Omega\cdot\nabla b^\be
&=dj-h\omega
+\partial_{\be} u^{\be^\perp}\cdot\nabla_{\be^\perp}^{\perp} b^\be
-\partial_{\be} b^{\be^\perp}\cdot\nabla_{\be^\perp}^{\perp} u^\be       \\
&\quad +2\partial_{\nv} b^\be\partial_{\ta} u^\be
-2\partial_{\ta} b^\be\partial_{\nv} u^\be .
\end{aligned}
\end{equation}
The $m=\be$ term in the Jacobian sum in \eqref{eq:appendixCurrentStart} is
\[
2\bigl(\partial_{\ta} b^\be\partial_{\nv} u^\be-
\partial_{\nv} b^\be\partial_{\ta} u^\be\bigr),
\]
which cancels the last line of \eqref{eq:stretchingCurrent}.  Thus
\[
\partial_t j+u\cdot\nabla j-b\cdot\nabla\omega-\eta\Delta j
=dj-h\omega+
\partial_{\be} u^{\be^\perp}\cdot\nabla_{\be^\perp}^{\perp} b^\be
-\partial_{\be} b^{\be^\perp}\cdot\nabla_{\be^\perp}^{\perp} u^\be
+2\sum_{\ell\in\{\ta,\nv\}}J_\be(b^\ell,u^\ell),
\]
which is the third equation in Lemma~\ref{lem:movingScalarSystem} without the
frame commutator.

\subsection{The magnetic pressure contribution in the \texorpdfstring{$d$}{d} equation}\label{app:pressureSplit}
The $\be$-component of the velocity equation is
\[
\partial_tu^\be+u\cdot\nabla u^\be-b\cdot\nabla b^\be-
\mu\Delta u^\be+\partial_{\be} P=0.
\]
Applying $\partial_{\be}$ gives
\[
\partial_td+u\cdot\nabla d-b\cdot\nabla h-
\mu\Delta d
= -\partial_{\be} u\cdot\nabla u^\be+
\partial_{\be} b\cdot\nabla b^\be-\partial_{\be}^2P.
\]
The pressure satisfies
\[
  -\Delta P=Q_u-Q_b,
  \qquad
  Q_u=\sum_{\ell,m\in\{\ta,\nv,\be\}}\partial_\ell u^m\partial_m u^\ell,
  \qquad
  Q_b=\sum_{\ell,m\in\{\ta,\nv,\be\}}\partial_\ell b^m\partial_m b^\ell .
\]
Thus the magnetic part of $-\partial_{\be}^2P$ is, up to the sign convention for
$(-\Delta)^{-1}$,
\[
  \partial_{\be}^2(-\Delta)^{-1}Q_b.
\]
The direct magnetic term decomposes as
\[
  \partial_{\be} b\cdot\nabla b^\be
  =h^2+\partial_{\be} b^{\be^\perp}\cdot\nabla_{\be^\perp} b^\be,
\]
and the pressure quadratic form decomposes into horizontal-horizontal,
horizontal-vertical, and vertical-vertical pieces:
\begin{equation}\label{eq:appendixQbSplit}
  Q_b
  =\sum_{\ell,m\in\{\ta,\nv\}}\partial_\ell b^m\partial_m b^\ell
  +2\sum_{\ell\in\{\ta,\nv\}}\partial_\ell b^\be\partial_{\be} b^\ell
  +h^2.
\end{equation}
This is the split used in \eqref{eq:pressureSplitB}.  The first term in
\eqref{eq:appendixQbSplit} is estimated by the horizontal Hodge control of
$b^{\be^\perp}$, while the latter two terms are either algebraic or of the mixed
anisotropic type covered by Lemma~\ref{lem:LZproduct}.

\section{Finite jumps of the moving direction}\label{app:jumps}

Let the jump times of $\be$ in $[0,T]$ be
\[
  0=t_0<t_1<\cdots<t_N<t_{N+1}=T.
\]
On each open interval $(t_k,t_{k+1})$, $0\le k\le N$, we choose a moving frame
satisfying \eqref{eq:frame}--\eqref{eq:frameDerivative}.  The estimates of
Sections~\ref{sec:vorticity_current}--\ref{sec:moving_apriori} and
Subsection~\ref{sec:H1_closure} are local in time on that interval.  At $t_k$
the strong solution is continuous in $H^1$, hence
\[
\F(t_k^+)
\lesssim \|\nabla u(t_k)\|_2^2+\|\nabla b(t_k)\|_2^2<\infty,
\]
with a constant independent of the new orthonormal frame.  Therefore the
argument can be restarted at $t_k$ with the new frame.  Since $N<\infty$, a
finite induction propagates the bounds \eqref{eq:Fbounded},
\eqref{eq:Dintegrable}, and \eqref{eq:H1Bound} through all jumps.

\section*{Declaration of generative AI and AI-assisted technologies in the manuscript preparation process}
During the preparation of this work the authors used OpenAI's ChatGPT to assist
with manuscript editing, organization, and language refinement.  After using this
tool, the authors reviewed and edited the content as needed and take full
responsibility for the content of the publication.

\end{document}